\documentclass[twoside,a4paper,11pt]{article}
\usepackage{amsfonts, amsbsy, amsmath, amsthm, amssymb, latexsym, enumerate,verbatim}
\usepackage{graphicx}

\usepackage{tikz} 
\usetikzlibrary{positioning}

\newtheorem{propo}{Proposition}[section]

\newtheorem{lemma}[propo]{Lemma}
\newtheorem{corol}[propo]{Corollary}

\newtheorem{theo}[propo]{Theorem}
\newtheorem{examp}[propo]{Example}
\newtheorem{rem}[propo]{Remark}

\theoremstyle{definition}

\newcommand{\ld}{,\ldots ,}

\newcommand{\ra}{ \rightarrow }

\newcommand{\lan}{ \langle }
\newcommand{\ran}{ \rangle }

\newcommand{\diag}{\mathop{\rm diag}\nolimits}

\newcommand{\Id}{\mathop{\rm Id}\nolimits}
\newcommand{\Irr}{\mathop{\rm Irr}\nolimits}

\newcommand{\Om}{\Omega}

\newcommand{\CC}{\mathop{\mathbb C}\nolimits}
\newcommand{\FF}{\mathbb {F}}

\newcommand{\ZZ}{\mathop{\mathbb Z}\nolimits}

\newcommand{\al}{\alpha}
\newcommand{\be}{\beta}

\newcommand{\ep}{\varepsilon}

\newcommand{\lam}{\lambda }
\newcommand{\om}{\omega }

\newcommand{\up}{^{-1}}

\newcommand{\si}{\sigma }

\def\11{$(1)$}
\def\22{$(2)$}
\def\33{$(3)$}

\def\d12{{_{12}}}

\def\ almost cyclic{{algebraically closed }}
\def\ almost cyclicf{{algebraically closed field }}
\def\ almost cyclicft{{algebraically closed field. }}

\def\au{{automorphism }}

\def\ei{{eigenvalue }}
\def\eis{{eigenvalues }}

\def\f{{following }}

\def\ho{{homomorphism }}
\def\hos{{homomorphisms }}

\def\ii{{if and only if }}

\def\ir{{irreducible }}

\def\irt{{irreducible. }}
\def\irr{{irreducible representation }}

\def\itf{{It follows that }}

\def\mult{{multiplicity }}

\def\po{{polynomial }}

\def\rep{{representation }}
\def\reps{{representations }}
\def\rept{{representation. }}

\def\St{{Suppose that }}

\newcommand{\bl}{\begin{lemma}\label}
\newcommand{\el}{\end{lemma}}
\newcommand{\bp}{\begin{proof}}

\def\ag{algebraic group }
\newcommand{\med}{\medskip}

\def\ac{almost cyclic }

\newcommand{\enp}{\end{proof}}

\def\st{Suppose that }

\date{}

\begin{document}

\title{Upper bounds for eigenvalue multiplicities\\
of almost cyclic elements in irreducible representations\\ of simple algebraic groups}
\author{A.E. Zalesski}
\maketitle


\begin{abstract} We study the rational \ir \reps of simple algebraic groups in which some non-central semisimple element
has at most one \ei of \mult greater than $1$. We bound the \mult
of this \ei in terms of the rank of the group.  
\end{abstract}

\bigskip
\centerline{{\it Dedicated to the memory of Irina Suprunenko}}
\bigskip

\section{Introduction}\label{sec:intro}

In this paper we study a special case of a general problem on \ei multiplicities of semisimple elements in \ir \reps of algebraic groups. All \reps of algebraic groups  and their modules   are assumed to be rational. This problem is connected with a similar problem on weight multiplicities.
In  \cite{TZ2} we determine the \ir \reps of simple algebraic groups $G$  whose all non-zero weights are of \mult 1, and compute the zero weight \mult in each case. Note that  weights are one-dimensional \reps of a maximal torus $T$, say,  of $G$,
and the \mult of a weight $\om$, say, in a \rep $\rho$ of $G$ is the \mult of  $\om$ as a constituent of $\rho(T)$.
Turning to individual non-central semisimple elements $g\in G $,  an analog of this result could be (i) obtaining a list of the \ir \reps $\rho$ of $G$ such that at most one \ei \mult of $\rho(g)$ is greater than 1, and then (ii) determining the \mult of the exceptional eigenvalue. An essential step of understanding problem (i) is made in \cite{TZ21}, where it is proved that  all non-zero weights of $\rho$ must be  of \mult 1. This does not mean however that the \ei multiplicities of  $\rho(g)$ 
are the same as the weight multiplicities. This can be seen on examples with $G$ to be of classical type and $\rho$ the natural \rep of $G$. Moreover, in these examples $m_g(\rho)$, the maximum \ei \mult of $\rho(g)$, is almost arbitrary (subject to the obvious condition $m_g(\rho)\leq \dim\rho$),
whereas all weight multiplicities of  $\rho$  are equal to 1. However, it is shown below that $m_g(\rho)$ can be bounded in terms of the rank of $G$. 
An initial step in this direction is made in \cite{TZ21}, where, for a non-central non-regular semisimple element $g\in G$, it is shown that the natural \rep of a classical group $G$ is essentially the only one in which at most one \ei \mult of $\rho(g)$ is greater than 1;  that result is quoted below as Theorem \ref{c99}. To be more precise, $\rho$ is a twist of the natural representation. (We call a \rep $\rho'$ a twist of $\rho$ if $\rho'(g)=\rho(\si(g))$ for $g\in G$ and a surjective \ho $\si:G\ra G$.) Therefore,
the current paper mainly deals with regular semisimple elements.

\def\hw{highest weight }

The main result of this paper is the following theorem:

\begin{theo}\label{ib2} Let $G$ be a simple  algebraic group of rank n  and $V$ an  \ir  $G$-module. Let $s\in G$ be a  regular semisimple element and let $m_s(V)$ be the maximum of the \ei multiplicities of $s$ on V. Suppose that at most one \ei \mult of $s$   on V exceeds $1$.  Then $m_s(V)\leq n+1.$ \end{theo}

In fact, we can give more details   as follows:

  \begin{theo}\label{ab9} Let $G$ be a simple  algebraic group of rank n  and $V$ an   \ir  $G$-module.  Let $s\in G$ be a regular  semisimple element and let $m_s(V)$ be the maximum of the \ei multiplicities of $s$ on V. Suppose that at most one \ei \mult of $s$   on V exceeds $1$   and  $m_s(V)>1$. Then either $m_s(V)=2$ or  one of \f holds. 

\med
$(1)$ $G=A_n$ and $m_s(V)\leq n+1;$ 

\med
$(2)$  V is a twist of a G-module from Table $1$ and $m_s(V)$ equals the weight zero \mult of V recorded there; in particular $n-2\leq m_s(V)\leq n+1;$

\med
$(3)$ $G=E_6$    and $m_s(V)\leq 3.$    \end{theo}

The bound in item (1) of Theorem \ref{ab9} is   sharp for $n>1$, at least for $p\neq 0$, see Example \ref{e2x}.
In this example for $G=A_n$ we construct a tensor-decomposable $G$-module $V$ with $m_s(V)=n+1$.  For $E_6$ see Example \ref{e65}.

\begin{corol}\label{hv2} 
Let G be a simple algebraic group  in defining characteristic $p\geq 0$, $s\in G$ a regular semisimple element  and let V be  an \ir G-module with \hw $\om\neq 0$.  Suppose that $\dim  V>2$ and at most one \ei \mult of $s$   on V exceeds $1$. Then 
$m_s(V)< \dim V/2$, unless $G=A_1$, $\dim V=3,p\neq  2$ or  $\dim V=4$ and either  $\om=3p^i\om_1, p  \neq 2, 3$ or  $\om=(p^i+p^j)\om_1$ with $1\leq i<j$. (If $p=0$ and $n=4$ then $\om=3\om_1$.)
\end{corol}

The fact that the \mult $m_s(V)$ in question is of certain significance can be observed from  \cite{GPPS}, in which 
the authors study \ir subgroups $X$ of $GL_n(q)=GL(V)$,   containing a semisimple element $g$ of prime order  acting irreducibly on $(\Id -g)V$. They assume that $\dim (\Id -g)V> \dim V/2$, equivalently, $m_s(V)=\dim V -\dim (\Id -g)V< \dim V/2$, and we complement this result by dropping the latter assumption for $X$  a quasi-simple group of Lie type 
in defining characteristic $p|q$. More precisely, we use Corollary \ref{hv2} to prove the \f result.   In the statement below $\overline{V}=V\otimes F$, where $F$ is an algebraically closed field  containing $\FF_q$.

\begin{theo}\label{hv3} 
Let G be a quasi-simple  group of Lie type in defining characteristic $p>0$,  
and let V be  a non-trivial  \ir$\FF_q G$-module where $p|q$. 
Let $s\in G$ be a regular semisimple element.  Suppose that $\dim V>2$ and  at most one \ei \mult of $s$   on $\overline{V}$ exceeds $1$. Then 
$m_s(\overline{V})< \dim V/2$, unless $G=SL_2(q')$ or $PSL_2(q')$ for some p-power $q'$, and either  $m_s(\overline{V})= \dim V/2$ is even or $\dim V=3,p\neq  2$ and $s^2\in Z(G)$.

Conversely, if V is a vector space over $\FF_q$ and $\dim V\equiv 0\pmod 4$ or $p\neq  2$ and   $\dim V=3$ then there exists
a group $G\cong SL_2(q')$, a semisimple element $s\in G$ and \irr $\rho:SL_2(q')\ra GL(V) $ such that at most one \ei \mult of $\rho(s)$ on $\overline{V}$ exceeds $1$   and $m_s(\overline{V})=\max\{\dim V/2, 2\}$. \end{theo}

In Theorem \ref{hv3}   saying that an element $g\in G$ is regular  means that $g$ is regular as an element of a simple algebraic group $\mathbf{G}$ from which $G$ is obtained as the fixed point set of a Steinberg (Frobenius) endomorphism  of $\mathbf{G}$, see \cite[\S 21.1]{MT}. For a semisimple element $g\in G$ this is equivalent to saying that $C_G(g)$ has no non-identity unipotent element.  If $g\in G$ is semisimple and  not regular then we have a result similar to that of \cite{GPPS}. To state it, it is convenient to introduce a notion of a nearly natural module of a classical group as follows.

Let $G$ be a quasi-simple  group of classical Lie type in defining characteristic $p>0$. Let $V$ be an \ir $\FF_qG$-module for some $p$-power $q$ and $F$ an algebraically closed field containing  $\FF_q$, and 
$\overline{V}=V\otimes F$. Let $G=\mathbf{G}^\si=\{g\in \mathbf{G}:\si (g)=g\}$, where $\si$ is a Steinberg endomorphism of $\mathbf{G}$. We say that $V$ is {\it nearly natural} if the composition factor dimensions of $G $ on $\overline{V} $ are the same as the dimension of the natural $\mathbf{G}$-module. 
(See {\it Notation} for the formal definition of the natural module.) 
This notion unfortunately depends on the Lie type of $\mathbf{G}$. 
For instance, the dimensions of the natural module for $A_3\cong D_3$
are 4 and 6, respectively, and these for $B_2\cong C_2$ with $p\neq 2$ are 5,4, respectively.

The following statement is a special case of the main theorem of \cite{GPPS}. 

\begin{theo}\label{gpps} Let   $G$ be a quasi-simple  group of  Lie type in defining characteristic $p>0$ and $q$ a $p$-power.  Let $s\in G$ be a  non-central  semisimple element and  let $\rho:G\ra GL_n(q)$, $n>9$, be an \irr such that $\rho(g)$ is similar to $\diag(\Id_m,M)$, where  $M\in GL_{n-m}(q)$ is \irt Suppose that $m<n/2$. Then G is classical and $\rho$ is 
 nearly natural.  
\end{theo}

    Note that the cases with $n\leq 9$ are listed in \cite[Table 6]{GPPS}.

 Theorem  \ref{hv3} deals with regular elements  $g\in G$. The case where $g$ is not regular follows from Theorem \ref{c99} obtained in \cite{TZ21}:

\begin{theo}\label{nr2} Let G   be  a quasi-simple group of Lie type in defining characteristic $p$ and $g\in G$ a non-central semisimple element. Let $\rho:G\ra GL_n(q)$ be an \ir \rept Suppose that g is not regular and  at most one \ei \mult of $\rho(g)$ exceeds $1$. Then $\rho$ is nearly natural for $G$  
or for a group $G_1\cong G$.  (Here $G_1=Spin^+_6(q)$ if $G=SL_4(q)$; $G_1=Spin^-_6(q)$ if $G=SU_4(q)$;   $G_1=C_2(q)$ if $G=B_2(q)$;  $G_1=B_2(q)$ if $G=C_2(q)$.) 
\end{theo}


\section{ Notation} The cardinality of a finite set $A$ is denoted by $|A|$.
The ring of integers is denoted by $\ZZ$, and by $\ZZ^+$ the subset of non-negative integers. If $0\neq a,b$ are natural numbers then we write $a|b$ to say that $b$ is a multiple of $a$ and we denote by $(a,b)$ the greatest common divisor of $a,b$.    We fix $F$ for an algebraically closed field of characteristic $p\geq 0$; the multiplicative group of $F$ is denoted by $F^\times$. The field of complex numbers is denoted by $\CC$. We write $Z(X)$ for  the center of a group $X$.
For $x\in X$ the group generated by $x$ is denoted by $\lan x \ran$, by $C_X(x)$ the centralizer of $x$ in $X$ and by $|x|$ the order of $x$. A diagonal $(n\times n)$-matrix with sequent diagonal entries $d_1\ld d_n$ is denoted by $\diag(d_1\ld d_n)$.

For our convenience we call an element $g\in GL_n(F)$ {\it cyclic} if $g$ is diagonalizable and all \eis of $g$
are of \mult 1. This is equivalent to saying that the characteristic \po of $g$ has no multiple roots. We call  $g$ {\it almost cyclic} if $g$ is diagonalizable and at most one of the \eis of $g$
is of \mult greater than 1. 
If $g\in GL_n(F_1)$, where $F_1$ is a subfield of $F$ then saying that $g$ is almost cyclic
means that $g$ is almost cyclic in $GL_n(F)$.
(See \cite{DZ1} and \cite{DZ3} for more comments on these notions.)

In Section 3  $G$ is a non-trivial  simply connected simple linear algebraic group  over $F$ 
 (unless otherwise explicitly stated); sometimes we write $G(F)$ to make it precise. (In Section 5 we use $\mathbf{G}$ for a simple algebraic group and $G$ for a finite group of Lie type.)
 All $G$-modules considered are rational finite-dimensional $FG$-modules. For a $G$-module $V$ (or a rational \rep $\rho$ of $G$)
we write $V\in\Irr(G)$ (or $\rho\in\Irr(G)$) to mean that $V$ (or $\rho$)  is irreducible. If $H$ is a subgroup of $G$ then $V|_H$ means the restriction of a $G$-module $V$ to $H$.

We fix a maximal torus $T$ in $G$ and a Borel subgroup containing $T$, which in turn define the roots and simple roots of $G$, as well as the weights of
$G$-modules and representations. The set $\Omega={\rm Hom} (T,F^\times)$ (the rational homomorphisms of $T$ to  $F^\times$)
is called the {\it weight lattice}, which is a free $\ZZ$-module of finite rank
called {\it the rank of} $G$. The elements of $\Omega$ are called {\it $T$-weights of} $G$. As $T$ is fixed we write "weights" in place of  $T$-weights.
For $t\in T$ and $\mu\in \Om(V)$, $\mu(t)$ is an \ei of $t$ on $V$; then $(-\mu)(t)=t\up$
and $(\mu+\nu)(t)=\mu(t)\nu(t)$ for $\nu\in\Om(V)$.

The $T$-weights of a $G$-module $V$ are the \ir constituents of the restriction of $V$ to $T$.
When $T$ is fixed (which we assume throughout the paper), we omit the indication to $T$ and write "weights" in place of "$T$-weights". The set of weights of $V$ is denoted by $\Omega(V)$. If $\mu$ is a weight of $V$, the $\mu$-weight space
is $\{v\in V: tv=\mu(t)v$ for all $t\in T\}$; the dimension of the $\mu$-weight space is called the {\it \mult of $\mu$
on} $V$.  The Weyl group of $G$ is denoted by $W(G)$; in fact $W(G)=N_G(T)/T$ so the conjugation  action of $N_G(T)$
on $T$ yields an  action of $W(G)$ on $T$ and on the $T$-weights. The $W(G)$-orbit of $\mu\in\Om$ is denoted by $W\mu$.
 
As every semisimple element of $G$ is conjugate to an element of $T$, in the most of our results the assumption $g\in G$ semisimple is equivalent to saying $g\in T$, the latter is  more rigorous when we deal with $\mu(g)$ for a weight $\mu$.

We use $G_{reg}$ to denote the set of semisimple regular elements  of $G$, in particular, $T_{reg}$ is the set of regular elements  of $T$. The semisimple regular elements  $g$ of $G$ are exactly those centralizing  no non-identity unipotent element  (Lemma 3.8), and see \cite[Introduction]{TZ22} for more comments on such elements and their properties.

With an algebraic group is associated the Lie algebra of $G$ denoted here by $L(G)$.  The group
$G$  acts on $L(G)$ in the natural way making it a $G$-module. If $G$ is simple (or reductive)
then the non-zero weights of $T$ on $L(G)$ are called {\it roots}. We denote the set of roots by $\Phi$ or $\Phi(G)$. The
$\ZZ$-span of $\Phi$ is called the {\it root lattice} and denoted by $R$ or $R(G)$. In $\Phi(G)$ one
can choose a $\ZZ$-basis  $\Pi$ of $R$ whose elements  are denoted by $\al_1\ld \al_n$ and called {\it simple roots} (here $n$ is the rank of $G$). The ordering of simple roots and the Dynkin diagrams of $G$ are as in \cite{Bo}.
 The weights in $R$ are called {\it radical}.

One defines a non-degenerate symmetric bilinear form on $\Omega$, which we express as $(\mu,\nu)$.
The elements $\om_i$  satisfying $(\om_j,\al_i)=2(\al_i,\al_i)\delta_{ij}$ (the Kronecker delta) belong to $\Omega$
and called {\it fundamental weights} \cite[Ch. VI,\S 1, no.10]{Bo}. These form a $\ZZ$-basis of $\Omega$ so every $\nu\in \Omega$ can be expressed in the form $\sum a_i\om_i$ with $a_1\ld a_n\in \ZZ;$ the set of  $\nu$ with $a_1\ld a_n\geq 0$ is denoted by $\Omega^+$.
The weights of $\Omega^+$ are called {\it dominant}. Suppose that $p>0$; a (dominant) weight $(a_1\ld a_n)$ is called $p$-{\it restricted} if $0\leq a_i<p$ for all $i=1\ld n$; the set of them is denoted by $\Omega^+_p$. For uniformity,
we often do not separate the cases with $p=0$ and $p>0$; by convention, a $p$-restricted weight is simply a dominant weight when  $p=0$. An \ir $G$-module is called $p$-{\it restricted} if so is the highest weight of it.

The weight with  $a_1=\cdots =a_n=0$ is denoted by $\om_0$ and called the {\it zero weight} (also the weight 0).
We often simplify $\mu\neq \om_0$ by writing $\mu\neq 0$.
 We can (and sometimes do) write the elements of $\Omega$ as $(a_1\ld a_n)$
and in some cases, especially when we compare \reps of $G(\CC)$ with $G(F)$,  saying that weights of $G(\CC)$ and $G(F)$
are the same means the coincidence of the set of strings $(a_1\ld a_n)$ for these groups.
We set $\Omega^+(V)=\Omega^+\cap \Omega(V)$, so $\Omega^+(V)$ is the set of dominant weights of $V$.

There is a standard partial ordering of elements  of $\Omega$; specifically, for $\mu,\mu'\in\Omega$ we write
$\mu\preceq\mu'$ and $\mu'\succeq \mu$ \ii $\mu'-\mu\in R^+$. (Sometimes we write $\mu\prec \mu'$ and $\mu'\succ \mu$ to indicate that $\mu\neq \mu'$.)
Every  non-trivial \ir $G$-module has a weight $\om$ such that $\mu\prec\om$ for every $\mu\in\Omega(V)$ with $\mu\neq \om$.
This is called the {\it highest weight of} $V$. There is a bijection between $\Omega^+$ and $\Irr (G)$,
so for $\om\in\Omega^+$ we denote by $V_\om$ the \ir $G$-module with highest weight $\om$.
The zero weight space of $V$ is sometimes denoted by $V^T$. In general, for $S\subseteq G$ we set $V^S=\{v\in V:
sv=v$ for all $s\in S\}$.   $L(G)$  is a $G$-module called {\it adjoint}; it is \ir if $p=0$; if $p>0$ it may be reducible but has a \hw denoted here by $\om_a$. 
 We denote by  $V_a$ the composition factor of $L(G)$ with \hw   $\om_a$.

Our notation for simple algebraic groups is standard, in particular those of types $A_n,B_n,C_n,D_n$ are called {\it classical}, $E_6,E_7,E_8,F_4,G_2$ are called {\it exceptional}, the subscript is the rank of the group.  Here $n\geq 1$ for $A_n$, and we usually assume $n>1$ for $B_n,C_n$, and $n>3$ for $D_n$, due to isomorphisms $D_3\cong A_3,$ $D_2\cong A_1\times A_1$ and $A_1\cong B_1\cong C_1$. For brevity we write $G=A_n$ to say that $G$ is a simple simply connected algebraic group of type $A_n$, and similarly for the other types. There is also an isomorphism $B_2\cong C_2$, and we usually choose $C_2$ for the standard representative of this algebraic group. (In Section 4 we allow to view $A_3$ as $D_3$ for uniformity of some statements.) If $p=2$ then for
every $n$ there  are surjective algebraic group \hos $B_n\ra C_n$ and $C_n\ra B_n$ with trivial kernel (so these  are abstract group isomorphisms); for our purpose the choice between these two groups is not essential, so we choose to work with $C_n$.

For an algebraic group $G$ of classical type the module (and the representation) with highest weight $\om_1$ is called {\it natural}. (We keep this definition even for types $B_n$ with $p=2$.)

In our reasonings we regularly use so called "Bourbaki weights", which are elements of a $\ZZ$-lattice containing $\Omega(G)$ with basis   $\ep_1,\ep_2,...$;  the explicit expressions of the fundamental weights and the simple roots of $G$ in terms of $\ep_i$'s is given in \cite[Planches 1 -- XIII]{Bo}.

 If $\si: G\ra G$ is a surjective algebraic group homomorphism  and $\phi$ is a representation of $G$ then the $\si$-twist $\phi^\si$ of $\phi$ is defined as the mapping $g\ra \phi(\si (g))$ for $g\in G$. Of fundamental importance
 is the Frobenius mapping $Fr: G\ra G$ arising from the mapping $x\ra x^p$ $(x\in F)$ when $p>0$. If $V$
 is a  $G$-module then the modules $V^{Fr^k}$  ($k\geq 1$ an integer) are called {\it Frobenius twists of} $V$; if $V$ is \ir with highest weight $\om$ then the highest weight of
 $V^{Fr^k}$ is $p^k\om$ for some integer $k\geq 0$. 

Every $\om=\sum a_i\om_i\in\Omega^+$ has a unique $p$-adic expansion $\om=\lam_0+p\lam_1+\cdots +p^k\lam _k$ for some integer $k\geq 0$, where $\lam_0\ld \lam_k\in\Omega_p^+$ are $p$-restricted dominant weights. By Steinberg's tensor product theorem \cite[Theorem 41]{St}, if $G$ is simply connected then    $V_\om=V_{\lam_0}\otimes V_{\lam_1}^{Fr}\otimes\cdots \otimes V_{\lam_k}^{Fr^k}$, which we refer as the Steinberg expansion of $V$.

A surjective algebraic group homomorphism $\si: G\ra G$ is   called {\it  Steinberg} if the subgroup $G^\si=\{g\in G: \si(g)=g\}$ is finite. Equivalently, $\si$ is Steinberg \ii $\si^m=Fr^k$ for some integers $m,k>0$, see \cite[p. 183]{MT}.
If $G$ is  semisimple and $\si$ is Steinberg then $G^\si$ is called a finite group of Lie type.

Let $G$ be an algebraic group, $V$ a $G$-module and let  $T$ be a maximal torus of $G$.  We say that $s\in T$  {\it separates a weight $\om$ of}  $V$ if the multiplicity of the \ei $\om(s)$ equals the \mult of $\om$.  In addition, we say that  $s$  {\it separates the weights  of}  $V$ if every eigenspace of  $s$ on $V$  is a weight space.  It is well known that $s$ is regular \ii it separates weight zero of the adjoint module. If $s$ separates all weights of the adjoint module, we say that $s$ is {\it strictly regular}.

The set of \eis of an element $g\in G$ on a $G$-module $V$ (disregarding the multiplicities)  is denoted by $E_V(g)$. The notation $m_s(V)$ is introduced in Theorem 1.2.

\section{Preliminaries}\label{prelims}

In this section $G$ is a simple simply connected linear algebraic group   over an algebraically closed field   of characteristic $p\geq 0$. For our purposes we can assume that $G=G(F)$, the group of $F$-points of $G$, where $F$ is an algebraically closed field of characteristic $p$.

Recall that a diagonalizable matrix is called \emph{cyclic} if all of its eigenspaces are one-dimensional, and \emph{almost cyclic}
if at most one of its eigenspaces has dimension greater than $1$.

In Lemma \ref{tt4} we denote by ${}^tM$  the transpose of a matrix $M$.

\begin{lemma}\label{tt4} 
  
$(1)$ Let $M\in M_{r\times r}(F)$ be an almost cyclic diagonal matrix such that $M$ is conjugate to ${}^tM^{-1}$ and let
 $  e\in F$ be an eigenvalue of $M$. If $e\notin\{ \pm1\}$, then the \mult of $e$ equals $1$.
  
$(2)$ Let $G$ be a semisimple algebraic group, $s\in G$ a semisimple element, and V  a G-module.    Suppose that $V$ is self-dual and  $s $ is  almost cyclic on $V$.    Then all eigenvalues of $s$ on $V$, not equal to $1$ or $-1$, occur with multiplicity $1$.
   In particular, this holds  if $V$ is irreducible and $G$
 is   simple of type $A_1, B_n, C_n, D_n$ ($n$ even), $E_n$, $F_4$ or $G_2$.\end{lemma}

\begin{proof} The first claim is clear, whereas the second one follows from a well known fact that every \irr
  of $G$ listed is self-dual.\end{proof}



\begin{lemma}\label{ma2} Let $M=M_1\otimes M_2$ be a Kronecker product of diagonal non-scalar matrices $M_1, M_2$ of sizes $m\leq n$,
respectively.

$(1)$ The \ei multiplicities of $M$ do not exceed $mt$, where t is the maximal \ei \mult 
of $M_2$. 

$(2)$ If $M_1,M_2$ are cyclic then the \ei multiplicities of $M$ do not exceed  $m$.

$(3)$ Suppose that M is almost cyclic. Then $M_1$ and $M_2$ are cyclic.

$(4)$ Suppose that M is almost cyclic  and $M_i$ is similar to $M_i\up$ for $i=1,2$. Then the \ei multiplicities of $M$ do not exceed  $2$. In addition, if $e$
is an \ei of M of \mult $2$ then $e\in\{\pm 1\}$. \el

\begin{proof} (1) Let $e$ be an \ei of $M$. Then $e=a_ib_j$, where $a_i,b_j$ are some \ei of  $M_1,M_2$, respectively. Then the \mult of $e$ equals $\sum m(a_i)m(b_j)$, where $m(a_i),m(b_j)$ is the \mult of $a_i,b_j$, respectively, and the sum ranges over the pairs
$a_i,b_j$ with $e=a_ib_j$.  We have $\sum m(a_i)m(b_j) \leq (\sum m(a_i))\cdot \max_j m(b_j)\leq m t$, as claimed. 

(2) Suppose the contrary, and let $r>m$ be the \mult of some \ei $e$ of $M$.
 Let $a_1\ld a_m$, $b_1\ld b_n$ be the \eis of $M_1$, $M_2$, respectively.
Then there exist $a_i,b_j,b_k$ $(1\leq i\leq m, 1\leq j<k\leq n)$ such that $a_ib_j=a_ib_k$ and hence $b_j=b_k$. This is a contradiction.

(3) Suppose that $M_1$ has an \ei $a$, say, of \mult  $r>1$. Let $b_1,b_2$ be distinct \eis of $M_2$. Then $ab_1,$ $a b_2$ 
are distinct \eis of $M$, each of \mult greater than 1. So $M$ is not almost cyclic, thus establishing the claim.

(4)   Suppose the contrary, and let $e$ be an \ei of $M$ of \mult at least 2. By (3), $M_1$ and $M_2$ are cyclic so $e=a_ib_i$ for $i=1,2$ and some (distinct) \eis $a_i$ of $M_1$ and $b_i$ of  $M_2$. Then $e\up=a_i\up b_i\up$ is an \ei of $M$ of the same \mult as that of $e$. As $M$ is almost cyclic,  $e=e\up$ and $e\in\{\pm 1\}$. Then $a_1b_1=a_2\up b_2\up$, whence
  $a_1b_2=a_2\up b_1\up$. 
If $a_2\up\neq a_1$ then  $a_1b_2$ is an \ei of $M$ of \mult at least 2; 
this is not equal to $e=a_1b_1$, hence $s $ is not almost cyclic on $M$, a contradiction. So $a_2= a_1\up$. If the \mult of $e$ is at least 3, say, $e=a_3b_3$ then, as above, this implies $a_3= a_1\up$. Then $a_3=a_2$, contrary to $(3)$. 
\end{proof}
 
Note that the assertion (1) of Lemma \ref{ma2} is known, see for instance \cite[Lemma 3.7]{LiS}.  

 \begin{lemma}\label{ss2} Let
 $V,M$ be non-trivial G-modules and  $s\in T$.

 $(1)$ If $s$ is  cyclic on $V$ then $s$ is regular.

 $(2)$ If $s$ is almost cyclic on $V\otimes M$ then $s$ is regular.\el

 \begin{proof} (1) Let $\rho$ be the \rep afforded by $V$. If $s$ is  cyclic on $V$ then
   $C_{GL(V)}(\rho(s))$ is a torus. Then so is $C_{G}(s)^\circ$, giving the claim.

   (2) By Lemma \ref{ma2}(3), $s$ is cyclic on $V$, so the result follows from (1).\end{proof}

In Lemma \ref{td2} below and elsewhere, for $A,B\subset \Om$ we define $A+B=\{a+b: a\in A,b\in B\}$. 

 \begin{lemma}\label{td2} Let
 $V,V_1, V_2$ be non-trivial G-modules.
Let $s\in G$ be a non-central semisimple element and assume that $s$ is    almost cyclic on V.

 $(1)$ \St $V=V_1\otimes V_2$. Then the matrices of $s$  on $V_1$ and $V_2$ are  cyclic, all weights of
$V_1$ and $V_2$ are of multiplicity $1$, and $s$ is regular.

$(2)$ Suppose that $\Om(V_1)+\Om(V_2)=\Om (V)$. Then $s$ separates the weights of $V_1$ and $V_2$.

$(3)$ Suppose that $\Om(V_1)+\Om(V_2)=\Om (V)$ and the weights of $V$ are of \mult $1$.
Then the \mult of every \ei of $s$ on $V$ does not exceed $|\Om(V_1)|$, the number of weights of $V_1$. \el

\begin{proof} Assertion (1) follows from Lemma \ref{ma2}, except the last claim, so suppose that $s$ is  not regular.  Then
  $C_G(s)$ contains a unipotent element $u\ne 1$. As $u$ stabilizes every eigenspace of $s$ on $V_1$, at least one of them is
  of dimension greater than $1$, contradicting that $s$ is cyclic on $V_1$.

  (2) Suppose the contrary, that the weights of $V_1$, say, are not separated by $s$, so
  there exist distinct weights
$\mu_1,\mu_2\in\Om(V_1)$ such that $\mu_1(s)=\mu_2(s)$.
Then for every $\lambda,\mu\in\Omega(V_2)$, $\mu_i+\lambda,\mu_i+\mu\in\Omega(V)$ for $i=1,2$ and
$(\mu_1+\lam)(s)=(\mu_2+\lam)(s)$ and $(\mu_1+\mu)(s) = (\mu_2+\mu)(s)$. This implies $\lam(s)=\mu(s)$ as  
$s$ is almost cyclic on $V$. Since $\mu$ is an arbitrary weight of $V_2$, we conclude that $s$ is scalar on $V_2$, and hence $s\in Z(G)$, a contradiction.

(3) Suppose the contrary. Let $e$ be an \ei of $s$ on $V$ whose \mult exceeds $m=|\Om(V_1)|$.
As all weight spaces of $V$ have dimension at most one, there are at least $m+1$ distinct weights
$\si_1\ld \si_{m+1}$ of $V$ such that $e=\si_1(s)=\cdots =\si_{m+1}(s)$. We have
$\si_{i}=\mu_i+\lam_i$ for $\mu_i\in\Om(V_1)$, $\lam_i\in\Om(V_2)$  for $i=1\ld m+1.$ By hypothesis,
$\mu_i=\mu_j$ for some
$1\leq i\ne j\leq m+1$, which yields both $\lambda_i\ne \lambda_j$ and  $\lambda_i(s) = \lambda_j(s)$, contradicting  (2).\end{proof}

Obviously, $\Om(V_1)+\Om(V_2)=\Om (V)$ if $V=V_1\otimes V_2$; however in many instances the former equality holds for $V\ne V_1\otimes V_2$, see Lemma \ref{wtlattice}(2) below. The next example illustrates the use of Lemma \ref{td2}(3).

\begin{examp} {\rm (1) Let $G=A_n$, $p=0$, $V_1=V_{\om_1}$, $V_2=V_{m\om_1}$, $V=V_{(m+1)\om_1}$.
Then all weight multiplicities of $V$ are known to be equal to 1 and $\Om(V_1)+\Om(V_2)=\Om (V)$
by {\rm \cite[Ch. VIII, \S 7, Proposition 10]{Bo8}}. So if $s$ is almost cyclic on $V$ then the \mult of any \ei of $s$ on $V$ is at most $n+1$.  

(2)  Let $G=A_n$, $V_1=V_{\om_1}$, $V_2=V_{\om_n}$ and $V=V_{\om_1+\om_n}$.  Then $\Om(V)=\Om(V_1)+ \Om(V_2)$. As $V$ satisfies the hypotheses of Lemma \ref{nm1}(2),   some $s$ is  almost cyclic on $V$. Here the zero weight of $V$ is of \mult $n$ if $n+1$ is not a multiple of $p$, and $n-1$ otherwise. So the bound in Lemma \ref{td2}(3) is close to the best possible one. In  Example \ref{e2x} below we show that the bound is in fact sharp.} \end{examp}

For reader's convenience we record here the \f 
 well known observations on the natural modules for classical groups. Note that we add
the case $n=3$ for $D_n$ for our convenience here.

\begin{lemma}\label{co1} {\rm \cite[Lemma 6]{TZ21}} Let $G\in \{A_n$, $B_n,p\neq2$, $C_n,n>1$, $D_n,n\geq 3\}$,
$V=V_{\om_1}$ and let $s\in G$ be a non-central semisimple element.

$(1)$ If $s$ is cyclic on $V $, then $s$ is regular.
For $G=A_n$ or $C_n$, the converse is true.

$(2)$  If $G=B_n$, then $s$ is regular \ii the \mult of the \ei $-1$ on $V$ is at most $2$ and the other \ei multiplicities
are equal to  $1$.

$(3)$ If  $G=D_n$, then $s$ is regular \ii the multiplicities  of the \eis $1$ and $-1$ on $V$
are at most $2$ and the other \ei multiplicities are equal to  $1$. In addition, if $s$
is not almost cyclic on $V$, then $s$  is not almost cyclic on $V_{\om_2}$.

$(4)$ If $s$ is  regular, then $s$ is almost cyclic  on $V$ unless $G=D_n,p\neq 2$ and $1,-1$ are \eis of
$s$ on V, each of \mult $ 2$. \el

\bl{dc22}  Let $G=  B_2$, $p\neq 2$, and let $s\in T_{reg}$. Then s is cyclic  on $V_{\om_2}$ and  almost cyclic  on  $V_{\om_1}$. \el

\begin{proof} This follows from  \cite[Lemma 3.10(1)]{TZ22}.
(Note that there is an algebraic group \ho $h:Spin_5(F)\ra Sp_4(F)$; under this \ho the natural module for $B_2$ has  highest weight  $\om_2$ as a $C_2$-module. The natural module for $C_2$ viewed as  a module for $B_2$ has highest weight $\om_2$.)  
 \end{proof}

The \f lemma is frequently used with no explicit reference, especially when $S$ consists of a single element.

\begin{lemma}\label{re3} {\rm \cite[Ch. II, \S 4, Theorem 4.1]{Spr}} Let G be a connected semisimple algebraic group, T a maximal torus  and $S\subset T$ a subset.  
Then the \f conditions are equivalent:

$(1)$ $C_G(S)^\circ=T;$

$(2)$ $C_G(S)$ consists of semisimple elements; 

$(3)$ $\al(S)\neq 1$ for every root $\al$ of $G$.\el

The \f theorem proved in \cite{TZ21} is of fundamental significance for this paper,
this is used many times, often without an explicit reference.

 \begin{theo}\label{ag8} {\rm \cite[Theorem 1]{TZ21}} Let $G$ be a simply connected simple  algebraic group  and $\phi$ a non-trivial \irr of $G$. Then
the following statements are equivalent:

$(1)$ there exists a non-central semisimple element $s\in G$ such that 
$\phi(s)$ is almost cyclic; 

$(2)$ All non-zero weights of $\phi$ are of \mult $1$.
\end{theo}

Recall the notation $e(G)$, which denotes the maximum of the squares of the ratios of the lengths of the roots in $\Phi(G)$. We will rely on the following important result \cite[Theorem 1]{Pr}:

\begin{theo}\label{premet}  Assume $p=0$ or $p>e(G)$.
Let $\lambda$ be a $p$-restricted dominant weight. Then 
$\Omega(V_\lambda)=\{w(\mu)\, |\, \mu\in \Omega^+, \mu\preceq\lambda,w\in W(G)\}.$\end{theo}

An application of Theorem \ref{premet} then gives:

\begin{lemma}\label{wtlattice} {\rm \cite[Lemma 2.6]{TZ22}}
  Assume $p=0$ or $p>e(G)$. Let $\lambda,\mu\in \Omega^+$  with $\lambda$ $p$-restricted, and $V_\lambda$,  respectively,
$V_\mu$ the associated irreducible $G$-modules. Then the following hold.

\med 

{\rm{(1)}} If $\mu\prec\lambda$ then
    $\Omega(V_\mu)\subseteq \Omega(V_\lambda)$.

{\rm{(2)}} If $\lambda+\mu$ is $p$-restricted then   $\Omega(V_{\lambda+\mu})=\Omega(V_\lambda\otimes V_\mu)=\Omega(V_\lambda)+\Om( V_\mu)$.

{\rm{(3)}} If $\lam\neq 0$ is radical then some root is a weight of $V_\lam$; otherwise $\Omega(V_\lam)$ contains some minuscule weight.
\end{lemma}

A dominant weight $\nu\neq 0$ is called minuscule if $W\nu=\Om(V_\nu)$, that is, the weights of $V_\nu$  form a single orbit under the Weyl group. The minuscule weights are well known and tabulated in \cite[Ch. VIII,\S 3]{Bo8}.

\begin{lemma}\label{nm1} Let V be a non-trivial $G$-module. Let $T'\subset T$ be the subset of all $t\in T$ that
  separate the weights of V.  
    
{\rm{(1)}} $T'$ is nonempty Zarisky open in T.

    {\rm{(2)}} Suppose that at most one weight of $V$ has multiplicity greater than $1$. Then all elements  of $T'$ are almost cyclic on V.
  \end{lemma}

  \begin{proof} $(1)$  Let $\mu,\nu$ be weights of $V$. Then $T_{\mu,\nu} :=\{x\in T\, |\, \mu(x)=\nu(x)\}$ is
    a Zarisky
    closed subset $T_{\mu,\nu}$ of $T$. The set of elements of $T$ that do not
    separate some pair of weights of $V$, being the finite union of all $T_{\mu,\nu}$ with $\mu\neq \nu$,
    is a proper closed subset of $T$. So $T'=T\setminus (\cup T_{\mu,\nu})$ is nonempty open.

    $(2)$ If $s\in T'$ then $\mu(s)\neq \nu(s)$ whenever $\mu\neq \nu$ are weights of $V$.
    Therefore, the \eis of
    $s$ on $V$ are exactly $\mu(s)$, where $\mu$ runs over the weights of $V$,
and the \mult of $\mu(s)$ equals that of $\mu$, whence the claim.  \end{proof}
 
We often use the \f lemma with no explicit reference.

\bl{ft1} $(1)$ Let V,$V_1$ be FG-modules (for some group G) that are  Galois conjugates of each other. Let $s\in G$ be a $p'$-element. Then $m_s(V)=m_s(V_1)$.

$(2)$ Let G be an \ag and $s\in G$ a semisimple element. Let V be a G-module and $V_1$ a Frobenius twist of V. Then $\dim V^s=\dim V_1^s$ and  $m_s(V)=m_s(V_1)$.
\el
 
\bp (1) is obvious. (2) A Frobenius twist, as it is defined in Section 2, arises from a so called standard Frobenius mapping $Fr:G\ra G$; the latter can be defined in terms of the mapping $h_k:(a_{ij})\ra (a^{p^k}_{ij}))$ for every matrix $(a_{ij})\in GL_n(F)$ as follows, see \cite[\S 1.17]{Ca}. A mapping $Fr:G\ra G$ is called standard Frobenius if   there exists a faithful rational \rep $\rho:G\ra GL_n(F)$ such that
   $\rho(Fr(g))=h_k(\rho(g))$ for every $g\in G$.  If $(a_{ij})$ is a diagonal matrix then  $h_k(a_{ij})=(a_{ij})^{p^k}$. As $s$ is semisimple, so is $\rho(s)$, and hence $x\rho(s)x\up$ is diagonal for some $x\in GL_n(F)$. So $h_k(x\rho(s)x\up)=h_k(x)\rho(Fr(g))h_k(x)\up)=(x\rho(s)x\up)^{p^k}=x\rho(s^{p^k})x^{-1}$, whence
$\rho(Fr(s))$ is a conjugate of $\rho(s^{p^k})=\rho(s)^{p^k}$ for every \rep $\rho$.  In particular, this holds for the \rep afforded by $V$. Since $\rho(s)$ is semisimple, the \ei 1 \mult of $\rho(s)^{p^k}$ coincides
with that of $\rho(s)$. If $e_1,e_2$ are distinct \eis of $\rho(s)$ then $e^{p^k}_1,e_2^{p^k}$ are distinct \eis of $\rho(s)^{p^k}$. So the result follows. (It is rather customary to use the term "standard Frobenius" only in the case with $k=1$; in the above reasoning it is convenient not to assume $k=1$.) \enp

Note that we cannot drop "Frobenius" from the assumption in (2); indeed, let $G$ be a simple \ag of type $D_4$, and  the highest weights of  $V,V'$ are $\om_1,\om_4$, respectively.  Then $V'$ is a  twist of $V$
(by a graph \au of order 3) but, if $p\neq 3,5$ then $m_s(V)=2$, $m_s(V')=0$ for a suitable element of order 15.

\bl{a0b} Let $G$ be a simple  algebraic group in characteristic p  and $V$ a non-trivial p-restricted \ir  $G$-module. Let $p=0$ or $p> e(G)$. 
If $0\in\Om(V)$ then one of the \f holds:

$(1)$  $\Phi(G)\subseteq \Om(V);$ 

$(2)$ $V$ is the short root module;

$(3)$ $G$ is of type $C_n$ and $\om=\om_i$ for $i>2$ even.  \el

\bp If all root lengths are the same then the result is known, see for instance \cite[Proposition 2.3]{SZ06}. 
So we are left with groups $G$ of type $B_n,C_n,F_4$ or $ G_2$, and we have to show that the maximal root $\om_a$ lies in $\Om(V)$ unless (2) or (3) holds. We proceed case-by-case.
Let $\om=\sum a_i\om_i$ be the \hw of $V$.

Let $G$ be of type $B_n$. Then $2\om_1\succ \om_a=\om_2$ and $2\om_n\succ \om_{n-1}\succ\om_{n-2}\succ  \cdots\succ \om_2\succ \om_1$, see \cite[Planche II]{Bo}. As $\om\succ 0$, the coefficient $a_n$ must be even. 
Therefore, $\om\succeq (a_1+\cdots +a_{n-1}+\frac{a_n}{2})\om_1$. If  $a_1+\cdots +a_{n-1}+\frac{a_n}{2}>1$ then $\om\succeq 2\om_1\succ \om_2$,  as desired.
If  $a_1+\cdots +a_{n-1}+\frac{a_n}{2}=1$ and $\om\neq \om_1$ then $\om\in\{\om_2\ld \om_{n-1}, 2\om_n\}$. So the result again follows.

Let $G$ be of type $C_n$. Then $ \om_a=2\om_1 $, $\om_i\succeq\om_2$  for $i$ even, and $\om_i\succeq \om_1 $ for $i$ odd. 
Therefore, $\om\succeq(\sum_{i\,\, odd}a_i)\om_1+(\sum_{i\,\, even}a_i)\om_2$. Here $\sum_{i\,\, odd}a_i$ is even as $\om\succ 0$; if this is  non-zero then the result  follows. Otherwise, since $2\om_2\succ 2\om_1$, the result again follows unless $\om=\om_i$ for $i$ even. Moreover, 
$2\om_1=\om_a$ is not a subdominant weight of $\om_i$ as one can conclude by comparing the $\al_1$-coefficients of the expansions of $\om_i$'s in terms of simple roots
in  \cite[Planche III]{Bo}.

Let $G$ be of type $F_4$. Then $\om_2\succ\om_3\succ \om_1=\om_a\succ  \om_4 $ and $2\om_4\succ\om_1$, see \cite[Planche XIII]{Bo}. 
So
$\om\succeq (a_1+a_2+a_3+a_4)\om_4\succeq 2\om_4\succ \om_1$, unless $a_1+a_2+a_3+a_4=1$. If $a_1+a_2+a_3>0$ then the result follows, otherwise $a_1=a_2=a_3=0$, $a_4=1$, that is, $\om=\om_4$ is a short root.

Let $G$ be of type $G_2$. Then $2\om_1\succ \om_a=\om_2\succ\om_1$ so 
$\om\succeq(a_1+a_2)\om_1$. If $a_2>0$ or $a_2=0,a_1\geq 2$ then the result follows, otherwise $\om=\om_1$, which is  a short root.\enp

\def\bb{by Lemma \ref} 

\def\Bb{By Lemma \ref}

\bl{hh7} Let G be a simple \ag and V a G-module. Let $s\in T_{reg}$. Suppose that $\mu(s)=\nu(s)$ for distinct weights $\mu,\nu\in\Om(V)$
and $\mu-\al,\nu-\al\in\Om(V)$ for some $\al\in\Phi(G)$. Then s is not almost cyclic on V.
\el

\bp Suppose the contrary. Observe that $\mu(s)=\nu(s)$ implies $(\mu-\al)(s)=(\nu-\al)(s)$. If $s$  is almost cyclic on $V $ then $(\mu-\al)(s)=\mu(s)$,
whence $\al(s)=1$.  \Bb{re3}(3), this is a contradiction as $s$ is regular. \enp

\bl{sr1} Let $s\in G=SL_n(F)$ be a strictly regular element. Then $s$ is cyclic on $V_{\om_2}$.  \el

\bp The weights of $V_{\om_2}$ are $\ep_i+\ep_l$ for $i<l$, so if $s$ is not cyclic on $V_{\om_2}$ then $(\ep_i+\ep_l)(s)=(\ep_j+\ep_k)(s)$ for some $i<l$, $j<k$ and $(i,l)\neq (j,k)$. \itf $(\ep_i-\ep_j)(s)= (\ep_k-\ep_l)(s)$. As $\ep_i-\ep_j$ with $i\neq j$ is a root, and $s$ is regular, we have $i\neq j$. Similarly, we have $k\neq l$. Then  $\ep_i-\ep_j$, $\ep_k-\ep_l$ are distinct roots, which is a contradiction as $s$ is strictly regular. 
\enp

 For reader's convenience we quote here \cite[Theorem 2]{TZ21}:

\begin{theo}\label{c99} Let $G$ be a simply connected simple  algebraic group defined over a field  of characteristic $p\geq 0$ and
let $s\in G$ be a non-regular non-central semisimple element.
Let $V$ be a non-trivial irreducible $G$-module.
Then  either $s$ is not almost cyclic on $V$
or
one of the \f holds:

{\rm {(1)}} $G$ is classical, that is, of Lie type $A_n,B_n\,\,(p\neq 2)$, $C_n$ or $D_n$ and $\dim V=n+1,2n+1,2n,2n$, respectively;

{\rm{(2)}} $G=B_n,p=2$ and $\dim V= 2n;$

{\rm{(3)}}  $G=A_3$ and $\dim V=6;$

{\rm{(4)}} $G=C_2\cong B_2$, $p\neq 2$ and $\dim V=5.$
\end{theo}

\section{Eigenvalue multiplicities: a proof of Theorem \ref{ab9}}\label{sec:bound}

In this section we bound the \ei multiplicities of a regular semisimple element $s\in G$ on a $G$-module $V$ provided
$s$ is almost cyclic  on $V$. Of course, in this case there is at most one \ei of \mult greater than 1, so we bound the \mult of this single eigenvalue.

\subsection{A uniform bound}

\bl{mk8} Let G be a simple \ag in   characteristic $p=0$ or $p> e(G)$, and let $0\neq \om\in\Om^+$. If $p>0$,
suppose that  $\om$ is not of the form $ p\om'$ for some $\om'\in\Om^+$. 
Let $V=V_\om$ and $\lam\in \Om(V)$. Then there is $\al\in \pm \Pi$ such that  $\lam-\al\in\Om(V)$.
\el

\bp If $p=0$ then the result is true by \cite[Ch. VIII, \S 7, Prop. 5]{Bo8}. If $p> e(G)$ and $\om$ is $p$-restricted then the result follows from that for $p=0$ and Lemma \ref{wtlattice}. Suppose that $\om$  is not $p$-restricted. Then $V=M\otimes V'$, where $M,V'$ are \ir  $G$-modules, and   $M$ is  $p$-restricted. 
Every weight $\lam$ of  $V$ is of shape $\mu+\nu$ with  $\mu\in \Om(M)$ and $\nu\in \Om(V')$. As $M$ is $p$-restricted, we have  $\mu-\al\in  \Om(V)$ for some $\al\in \pm \Pi$ by the above. Then
 $\lam-\al=(\mu-\al)+\nu\in \Om(V)$, as required. 
\enp

\begin{theo}\label{do1}  Let G be a simple \ag of rank n in   characteristic $p$.  Assume $p=0$ or $p>e(G)$. Let $S\subset T$ be a subgroup such that $C_G(S) = T$,  and $V$ an \ir $G$-module with $p$-restricted
  highest weight $\om\neq 0$. Assume in addition that some homogeneous component U  of $V|_S$  is  of dimension greater than $1$  and all other homogeneous component of $V|_S$ are $1$-dimensional. Then

$(1)$ if $ V^T\neq 0$, then  $V^S= V^T$ and $\dim V^S\leq n+1;$

$(2)$ $\dim U\leq 2n$. 
\end{theo}

\begin{proof}  By \cite[Theorem 5]{TZ21}, all non-zero weights of $V$ are of \mult 1. Therefore, $\omega$ appears in  Tables 1,2, and in particular $\dim V^T\leq n+1$. 
By Lemma \ref{re3} we have $\al(S)\neq 1$ for every $\al\in\Phi(G)$.

\med
(1) Here we have $V^T\subseteq V^S$.  Suppose the contrary, that $V^T\neq V^S$, so $\dim V^S>1$. Note that $V^S$ is a sum of weight spaces of $V$, so  
$\mu(S)=1$ for some non-zero weight of $V$. 

Suppose first that $\omega$ is a root.  Then all non-zero weights of $V$ are also roots. (This is well known.)  
Since  $\al(S)\neq 1$ for every $\al\in\Phi(G)$, we conclude that
 $V^S = V^T$, a contradiction.

 Next suppose that $\omega$ is not a root. By the above, $\mu$ is not a root. By replacing $S$ by a conjugate if necessary, we may assume that $\mu$ is dominant. By Lemma \ref{mk8}, $\mu+\beta$ is  a weight of $V$
for some $\beta\in \pm\Pi$. Then $\mu+\beta\ne\beta$, while $\mu+\beta$
and $\beta$ restricts to the same (non-zero) character of $S$. If $\beta$ is not a weight of $V$ then $C_G(S)=T$ implies $\be(S)=1$, a contradiction. So $\beta$ is not a weight of $V$.  
But $0\prec \omega$ so Lemma \ref{wtlattice}(3)
implies that some root is a weight of $V$. We deduce that $\Phi(G)$ has two root
lengths and either the long roots or the short roots are not weights of $V$. As $p> e(G)$ if $p>0$,  by Lemma \ref{a0b}, either $\om$ is a short root, a contradiction,
or $G=C_n$ and $\om=\om_i$ with $i$ even. In the latter case $\om_j\in \Om(V)$ for $i-j>0$ even, and these  are the only dominant weights of $\Om(V)$. So $\mu=\om_r$ with $i-r>0
$ even. Here $r>2$ as $\om_2$ is a root. Then $r-2\neq 0$ and $\om_r-\om_{r-2}=\ep_r+\ep_{r-1}\in\Phi(G)$. 
Set $\be=-\ep_r-\ep_{r-1}$. As $\be$ is a short root, it lies in the $W$-orbit of $\om_2=\ep_1+\ep_2$, and hence $\be\in\Om(V)$. However, by the above,
we have a contradiction.

\med
(2)  
In view of (1) we can assume that  $\dim V^T\leq 1  $, so all weights of $V$ are of \mult 1.  
If $U= V^S$ then the claim follows from (1). So we now
assume that $U\cap V^S=0$ hence $U$ is the sum of $T$-weight
spaces which  are of dimension $1$ by Theorem \ref{ag8}. Assume for a contradiction that there exist distinct $T$-weights
$\mu_1,\dots,\mu_{2n+1}$ such that $\mu_j|_S = \eta$ for some character   $\eta$ of $S$ and all $1\leq j\leq 2n+1$. For each $j$ there exists
$\gamma_j\in \pm\Pi$ such that $\mu_j+\gamma_j$ is a weight of $V$. Then there exists $i\ne j$ such that $\gamma_i=\gamma_j$. But then
$\mu_i+\gamma_i$ and $\mu_j+\gamma_j$ are distinct weights of $V$ and have equal restrictions to $S$.   By the hypothesis,
 this then means that $\mu_i|_S = \eta|_S =
(\mu_i+\gamma_i)|_S$, whence $\gamma_i(S)=1$, a contradiction.
\end{proof}

\subsection{$G\neq A_n$: Special cases}

A more detailed case-by-case analysis allows us to improve the bound  $\dim V_S(\eta)\leq 2n$ in Theorem \ref{do1} to
$\dim V_S(\eta)\leq n+1$. We prefer to assume $S=\lan s\ran$ to be cyclic, the general case is similar. 
We first perform this for groups $G\neq A_n$, and then consider the case of $G=A_n,n>1,$ 
which turns out to be more complex. We denote by $m_s(V)$ the maximum of \ei multiplicities of $s\in G$ on $V$. 
In this section   we obtain a sharp bound $m_s(V)\leq 2$ in a number of special cases. 

Recall  (Theorem \ref{ag8}) that the non-zero weights of an \ir $G$-module $V$ are of \mult 1 whenever a semisimple element $s\in G$ is almost cyclic on $V$. Such modules are classified in \cite{TZ2} and the highest weights of $V$, for $V$ $p$-restricted, are reproduced in Tables 1,2 at the end of this paper.

\begin{lemma}\label{o10} Let $G=B_n$, $p\neq 2$, or $C_n$, $p=2$, and  
let $s\in T_{reg}$.  Suppose that s is almost cyclic  on $V_{\om_n}$.
Then   $m_s(V_{\om_n})\leq 2$.\el

\begin{proof} The weights of $V$ are $\frac{1}{2}(a_1\ep_1 +\cdots+a_n \ep_n)$ for $G=B_n$ and $a_1\ep_1 +\cdots+a_n \ep_n$ for $G=C_n$,
where $a_1\ld a_n\in\{\pm 1\}. $    Let $\mu,\nu\in\Om (V)$ be distinct weights such that $\mu(s)=\nu(s)$. Note that $\pm\ep_i$, respectively, $\pm 2\ep_i$ for $i=1\ld n$ are roots of $G$ of type $B_n$, respectively, $C_n$. If $\mu\neq -\nu$ then, obviously,  there is $i$ such that $\mu-\al,\nu-\al\in\Om(V)$ for $\al\in \{\pm\ep_i:1\leq i\leq n\}$, respectively,  $\al\in \{\pm 2\ep_i:1\leq i\leq n\}$. This contradicts Lemma \ref{hh7}. 

Now suppose the contrary, and let  $m_s(V_{\om_n})> 2$. Then there are distinct weights $\lam,\mu,\nu$ of $V$ such that $\lam(s)=\mu(s)=\nu(s)$.
By the above, $\lam=-\mu$, $\lam=-\nu$ and $\nu=-\mu$. This is a contradiction.\end{proof}

\begin{lemma}\label{o11} Let $G=B_n$, $n>2$, $p\neq 2$, or $C_n$, $p=2$, and let  $V$ be a non-trivial  \ir p-restricted G-module of \hw $\om$.    
Let $s\in T_{reg}$.  Suppose that s is almost cyclic  on $V$.
Then  one of the \f holds:

$(1)$ $\om\in\{\om_1,\om_n\}$ and $m_s(V)\leq 2;$ 

$(2)$  $\om=\om_2$, $n>2$, and $n-1\leq m_s(V)\leq n;$ 

$(3)$ $G=B_n, \om=2\om_1$ and  $n\leq m_s(V)\leq n+1$.\el

\bp By Theorem \ref{ag8}, all non-zero weights of $V$ are of \mult 1, and hence, by Theorem \ref{ag8}     and Tables 1,2, 
 $\om\in \{\om_1,\om_n,\om_2,2\om_1\}$, the latter for $p\neq 2$. For $\om=\om_1$ see Lemma \ref{co1}, for  $\om=\om_n$ see Lemma \ref{o10}.  
If  $\om=\om_2$, $n>2$,  or $2\om_1$  with $p\neq 2$ then $\dim V^T\geq 2$ by Table 1. Then $m_s(V)=\dim V^s$, and $V^T=V^s$ by Theorem \ref{do1}. So the result follows from Table 1. \enp

The case $G=B_2\cong C_2$, $p\neq 2$, will be considered later in Lemma \ref{2bc}.

\bl{11d} Let $G$ be an \ag of type $D_n, n>3$, and let $V=V_\om$, where $\om\in\{\om_1,\om_{n-1},\om_n\}$. Let $\mu,\nu\in\Om(V)$. Then  $\mu-\al,\nu-\al\in\Om(V)$ for some root $\al$ of G, unless one of the \f holds:

$(1)$ $\mu=-\nu$ and either n is even or $\om=\om_1;$

$(2)$ n is odd, $\om\neq \om_1$  and $\mu=-\nu\pm \ep_i$ for some $i\in\{1\ld n\}$.
\el

\bp  Suppose that $\om=\om_1$. Then the weights of $V$ are $\pm\ep_1\ld \pm\ep_n$. We can assume that $\mu=\ep_i$ for $i\in\{1\ld n\}$. If $\nu=\ep_j$ $(1\leq j\leq n, j\neq i)$ then set $\al=\ep_i+\ep_j$. 
If $\nu=-\ep_j$, $j\neq i$, then set $\al=\ep_i-\ep_j$. Then $\mu-\al,\nu-\al\in\Om(V)$ in each case.

 Suppose that $\om=\om_n$ or $\om_{n-1}$. Then the weights of $V$ are  $\frac{1}{2}(\sum a_i\ep_i)$ with $a_i\in\{1,-1\}$, and the number of indices $i$ with $a_i=-1$ is even if $\om  =\om_n$, otherwise it is odd.  Let $\mu=\frac{1}{2}(\sum a_i\ep_i)$, $\nu=\frac{1}{2}(\sum b_i\ep_i)$ be weights of $V$, and let $L=\{i: a_i=b_i\}$. If $|L|>1$ then $\al=a_i\ep_i+a_j\ep_j\in \Phi(G)$  for $i,j\in L$, $i\neq j$ and   $\mu-\al,\nu-\al\in\Phi(G)$, as required. 

Suppose that $|L|=1$. Let $i\in L, $ $j\notin L$, $1\leq j\leq n$. Then $\mu+\nu=\pm \ep_i$ and $n$ is odd (otherwise $-\nu\in \Om(V)$ and $\mu-(-\nu)=\pm \ep_i$ is not a radical weight).   

Suppose that $|L|=0$. Then $\mu=-\nu$ and $n$ is even.   \enp



\begin{lemma}\label{oo1} Let $G=D_n,n>3,$ let $s\in T_{reg}$.  
Suppose that s is almost cyclic  on $V= V_{\om_i}$ for $i\in\{1,n-1,n\}$. Then $m_s(V)\leq 2$.\el

\bp Suppose the contrary, and let $\mu_1,\mu_2,\mu_3$ be distinct weights of $V$ such that $\mu_1(s)=\mu_2(s)=\mu_3(s)$. If $\mu_i-\al,\mu_j-\al\in \Om(V)$ for some root $\al$ and some distinct  $i,j\in\{1,2,3\}$ then the result follows from Lemma \ref{hh7}. 

Suppose that for every distinct $i,j\in\{1,2,3\}$
and any root $\al$ either $\mu_i-\al\notin \Om(V)$ or $\mu_j-\al\notin \Om(V)$. By Lemma \ref{11d}, this implies that $n$ is odd and $\mu_i+\mu_j=\pm \ep_{k(i,j)}$ with $k(i,j)\in\{1\ld n\}$  for every pair $i,j\in\{1,2,3\}$ with $i\neq j$. If $k(1,2)\neq k(1,3)$ then $\mu_3-\mu_2=(\mu_1-\mu_2)-(\mu_1-\mu_3)=\pm \ep_{k(1,2)}\mp \ep_{k(1,3)}$ is a root, a contradiction. So  $k(1,2)= k(1,3)$, and similarly, $k(1,2)= k(2,3)$. This means that $\mu_i+\mu_j=a(i,j)\ep_k$ for some fixed $k\in\{1\ld n\}$ and $a(i,j)\in\{1,-1\} $ for every pair $i,j\in\{1,2,3\}$ with $i\neq j$. By reordering of $\mu_1,\mu_2,\mu_3$ we can assume that $\mu_1+\mu_2=\mu_1+\mu_3$, whence  $\mu_2=\mu_3$, a contradiction.  \enp

\bl{tp7} Let $G=D_n,n>3,$ let M be an \ir G-module with \hw $\om\in\{\om_1,\om_n,\om_{n-1}\}$ and let  $V $  an arbitrary \ir G-module. Let $s\in T_{reg}$. Suppose that s is almost cyclic on $M\otimes V$. Then  $m_s(M\otimes V)\leq 2$.
\el

\bp Suppose the contrary, and let $\lam_i$, $i=1,2,3$ be distinct weights of $M\otimes V$ such that $\lam_i(s)=e\in E(M\otimes V)$.
Then $\lam_i=\mu_i+\nu_i$, where $\mu_i,\nu_i$ are weights of $M,V$, respectively. \Bb{ma2}(3), $s$ is cyclic on $M$ and $V$.
If  (*) there exists a root $\al$ of $G$ such that $\mu_i-\al, \mu_j-\al\in \Om(M)$ for some $i,j\in\{1,2,3\}$ then $\mu_i-\al+\nu_i, \mu_j-\al+\nu_j\in \Om(V)$, and  $(\mu_i-\al+\nu_i)(s)=( \mu_j-\al+\nu_j)(s)=e\cdot \al(s)\up$. So the result follows from Lemma \ref{hh7}. 

Suppose that (*) does not hold. In this case, by Lemma \ref{11d},   $n$ is odd and $\mu_i+\mu_j=\pm \ep_{k(i,j)}$ with $k(i,j)\in\{1\ld n\}$  for every pair $i,j\in\{1,2,3\}$ with $i\neq j$. Then we repeat the argument in Lemma \ref{oo1} to get a contardiction.  

\enp

\begin{lemma}\label{cn3} Let $G=C_3$, $p\neq 2$. Let $s\in T_{reg} $.  
Suppose that s is almost cyclic  on $V=V_{\om_3}$. Then $m_s(V)\leq2$. \el

\begin{proof} Note that every weight of $V$ has \mult 1 (see Table 2)
and  $\Om(V)=X\cup Y$, where $X=\{\pm \ep_1, \pm \ep_2,\pm \ep_3 \}$ and $Y=\{\pm \ep_1 \pm \ep_2\pm \ep_3 \}$. Suppose the contrary, and let $\mu_i$, $i=1,2, 3$, be distinct weights of $V$ such that $\mu_i(s)=e$. By  Lemma \ref{tt4}, $e\in\{\pm 1\}$. If $\mu_i\in  X$ for some $i$
then $(-\mu_i)(s)=e$, which is false as $X=\Omega(V_{\om_1})$ and then $s$ is not regular by Lemma \ref{co1}(1). So $\mu_i\in Y$ for every $i=1,2,3$. Then
 $\mu_i=a_{i1}\ep_1+a_{i2}\ep_2+a_{i3}\ep_3$ for $a_{ik}\in \{\pm 1\}$ for $1\leq i,k\leq 3$. By reordering $\mu_i$ we can assume $a_{11}=a_{21}$, and then $\mu_1-\al,\mu_2-\al\in\Omega(V)$ for $\al\in \{\pm 2\ep_1\}$. This contradicts Lemma \ref{hh7}. \end{proof}

\begin{lemma}\label{5n5} Let $G=C_{n}$, $n>2$, $p=3$, and $\om\in\{\om_{n-1},\om_n\}$. Let $s\in T_{reg}$ be  almost cyclic  on $V_{\om}$. 
Then $m_s(V)\leq 2$.\el

\bp The non-zero weights of $V$ are $\pm \ep_{i_1}\pm \cdots \pm \ep_{i_r}$ for $r\in\{1\ld n\}$, where $r-n$ is odd if $\om=\om_{n-1}$, and even if $\om=\om_n$.
  
 Suppose the contrary, let $m_s(V)> 2$ and let  $\lam,\mu\in\Om(V)$, $\lam(s)=\mu(s)=e$, $\lam\neq\mu$,
where   we can assume that $\lam\neq -\mu$. By  Lemma \ref{tt4}, $e\in\{\pm 1\}$. Set  $\lam=\sum a_i\ep_i$, $\mu=\sum b_i\ep_i$, where  $a_i,b_i\in\{0,1,-1\}$.

Observe that $\lam,\mu$ are non-zero. Indeed, otherwise, $\lam(s)=\mu(s)=1$ so $\dim V^s\geq 2$. Then $m_s(V)=\dim V^s$
and $\dim V^T=\dim V^s$ by Theorem \ref{do1}(1). However,  $\dim V^T=1$,
which is a contradiction.

 If $a_i=b_i\neq 0$ for some $i$ then $\lam-2a_i\ep_i$, $\mu-2a_i\ep_i \in\Om(V)$ and $(\lam-2a_i\ep_i)(s)=(\mu-2a_i\ep_i)(s)$. As $2a_i\ep_i$ is a root, we have a contradiction \bb{hh7}.

If $a_i=-b_i\neq 0$ for some $i$ then we replace $\lam $ by $-\lam$; as $e\in\{\pm 1\}$, we have $(-\lam)(s)=\lam(s)$, so we  arrive at the case considered above.  

So $a_ib_i=0$ for every $i=1\ld n$. 
 Using the Weyl group, we can assume that $\lam=\ep_1+\cdots+\ep_{r-1}-\ep_r$,
$\mu=\ep_{r+1}+\cdots+\ep_{r'}$ for some $0<r<r'\leq n$. Then $\lam+(\ep_r-\ep_{r+1}),\mu+(\ep_r-\ep_{r+1})\in\Om(V)$, and $(\lam+(\ep_r-\ep_{r+1}))(s)=(\mu+(\ep_r-\ep_{r+1}))(s)$. As above, this contradicts Lemma \ref{hh7}. \enp

\begin{lemma}\label{cn5} Let $G=C_{n}$, $n\geq 2,p>3$, and $\om\in\{\om_{n-1}+\frac{p-3}{2}\om_n, \frac{p-1}{2}\om_n\}$. Let $s\in T_{reg}$ be  almost cyclic  on $V_{\om}$. Then $m_s(V_\om)\leq 2$.\el

\bp The weights of $V_\om$ are $l_1\ep_1+\cdots+l_n\ep_n$, where $-(p-1)/2\leq l_1\ld l_n\leq (p-1)/2$, see \cite[proof of Lemma 1.6]{Z90}; moreover, if $\mu$ is of this form then $\mu$ is a weight of $V_{\om}$ with $\om\in\{\om_{n-1}+\frac{p-3}{2}\om_n, \frac{p-1}{2}\om_n\}$.  Let $\lam,\mu,\nu$ be three distinct  weights of $V_\om$  such that  $\lam(s)=\mu(s)=\nu(s)$. As  $(-\lam)(s)=(-\mu)(s)$, we have $\lam(s)=\mu(s)=\nu(s)\in\{\pm 1\}$. Let
 $$\lam=l_1\ep_1+\cdots+l_n\ep_n,\,\,\,\,\mu=m_1\ep_1+\cdots+m_n\ep_n, \,\,\,\,\nu=k_1\ep_1+\cdots+k_n\ep_n.$$

\Bb{hh7}, either $\lam-2\ep_i$ or $\mu-2\ep_i$ is not a weight of $V_{\om}$ for every $ i=1\ld n. $    
Observe that $\lam-2\ep_i$, respectively,  $\mu-2\ep_i$ is  a weight of $\Om(V_{\om})$ unless $l_i\leq 1-\frac{p-1}{2}$, respectively,  or $m_i\leq 1-\frac{p-1}{2}$, and similarly for the pairs $\lam,\nu$ and $\mu,\nu$.
So there are at least two pairs with $i$-th coefficients at most $1-\frac{p-1}{2}$. By reordering $\lam,\mu,\nu$ we can assume that $l_i,m_i\leq 1-\frac{p-1}{2}$.
Then  $-\lam,-\mu$ are weights $V_{\om}$, $(-\lam)(s)=(-\mu)(s)$ and $-l_i,-m_i\geq \frac{p-1}{2}-1$. Then the $i$-th coefficient of $-\lam-2\ep_i$ equals $-l_i-2\geq  \frac{p-1}{2}-3\geq - \frac{p-1}{2}$ for $p>3$. So  $-\lam-2\ep_i\in\Om(V_{\om})$; similarly,  $-\mu-2\ep_i\in\Om(V_{\om})$. Then Lemma \ref{hh7} yields   a contradiction. \enp

\bl{2bc} Let $G=B_2\cong C_2$, $s\in T_{reg}$  and let  $V$ be a non-trivial  \ir p-restricted G-module of \hw $\om$. If s is almost cyclic on $V$ then  $m_s(V)\leq 2$.
\el

\bp Let $G=C_2$. As $s$ is almost cyclic on $V$, $\om$ occurs in Tables 1,2. 
For those in Table 1 we have $\dim V^T=2$, and hence $m_s(V)=\dim V^s=\dim V^T$ by Theorem \ref{do1}; then  $\dim V^T=2$ by Table 1.
Those in Table 2 are $\om_1,\om_2$ and, for $p>3$, also  $ \om_1+\frac{p-3}{2}\om_2, \frac{p-1}{2}\om_2$. In the two last cases the result is contained in Lemma \ref{cn5}. The module $V_{\om_1}$ is the natural module for  $G$, and $m_s(V)=1$ by Lemma  \ref{co1}(1).  The module $V_{\om_2}$ is the natural module for $B_2\cong C_2$, so the result follows from Lemma \ref{co1}(2).  
\enp

\begin{rem}\label{bc4} {\rm  If $\dim V=4$ then $m_s(V)=1$.   Indeed, if  $G=C_2$ then either $V$ 
is a Frobenius twist of the natural module or $p=2$ and $V$ is a Frobenius twist of $V_{\om_2}$. The former case is ruled out as $s$ is regular (see Lemma \ref{co1}), the latter case can  be checked straightforwardly. (Say, let $s=\diag(a,b,b\up,a\up)\in Sp_4(q)$, where the diagonal entres are distinct;  then the \eis of $s$ on $V_{\om_2}$ are $a^{\pm 1}b^{\pm 1}$;
if $m_s(V)>0$ then $a^ib^j=1$ for some $i,j\in\{\pm 1\}$, which is a contradiction.) 
If $G=B_2$ then $\om=p^k\om_2$; 
as $V_{\om_2}$ is the natural module for $C_2$, it follows that $m_s(V)=1$ for a regular element $s\in G$.}
\end{rem}

\bigskip
\begin{center}
{\small
   \begin{tikzpicture} 
[node distance=0.5cm, every node/.style={circle,draw=black,align=center}, 
font=\sffamily\bfseries]
 \node(c){B}; 
\node(a)[above left =of c]{$\om_1$}  edge  [label= below right:{1}] [ -> ](c); 
\node(b)[below right =of c]{C} edge [label= above left:{3}] [<-](c);  
\node(d)[below right= of b]{D} edge    [<-] (b);
\node(e)[below right= of d]{F} edge [label= above left:{2}] [<-] (b);
\node(f)[above right= of d]{E} 
 [<-] (d);   
\node(g)[above right= of e]{H} edge  [label= above left:{2}] [<-] (e); 
\node(h)[above right= of f]{G} edge   [<-](f); 
\node(i)[below right= of h]{I} edge    [<-] (h);
\path [->] (g) edge (i);
\path [->] (f) edge (g);
\node(j)[below right= of g]{J} edge [label= above left:{4}]  [<-] (g);
\node(k)[below right= of i]{K} edge [label= above left:{4}]     [<-] (i);
\node(l)[below right= of k]{N} edge [label= above left:{3}] [<-] (k);
\node(m)[above right= of k]{M} edge  [<-] (k);
\node(i)[below right= of h]{I} edge  [label= above left:{2}] [<-] (h);
\node(p)[below right= of m]{P} edge  [label= above left:{3}] [<-] (m);
\node(q)[below right= of j]{L} edge [label= above left:{3}] [<-] (j);
\node(r)[above right= of p]{R} edge  [<-] (p);
\node(t)[above right= of r]{T} edge  [<-] (r);
\node(s)[below right= of l]{Q} edge [label= above left:{1}]   [<-] (l);
\node(o)[below right= of q]{O} edge [label= above left:{1}] [<-] (q);
\node(u)[below right= of p]{S} edge [label= above left:{1}]  [<-] (p);
\node(w)[below right= of r]{U} edge [label= above left:{1}]  [<-] (r);
\node(v)[below right= of t]{V} edge [label= above left:{1}]  [<-] (t);
\node(x)[below right= of w]{W} edge [label= above left:{3}]  [<-] (w);
\node(y)[below right= of v]{X} edge [label= above left:{3}]  [<-] (v);
\node(z)[below right= of y]{Y} edge [label= above left:{4}]  [<-] (y);
\node(z1)[below right= of z]{Z} edge [label= above left:{5}]  [<-] (z);
\node(z2)[below right= of z1]{A} edge [label= above left:{6}]  [<-] (z1);

\node(f)[above right= of d]{E} edge [label= above left:{}] [<-](d);

\path [->] (p) edge (r);  
\path [->]   (j) edge (k);  
\path [->] (q) edge (l);
\path [->] (l) edge (p);
\path [->] (o) edge  (s);
\path [->] (x) edge (y); 
\path [->] (s) edge (w); 
\path [->] (w) edge (v); 
\begin{scope}[nodes = {draw = none}]
\path  (d) -- (f)  node [above right = 6pt] {6};
\path (c) -- (b)  node [below right = 6pt] {4};
\path (b) -- (e)  node [above right = 6pt] {5};
\path (h) -- (i)  node [below left = 9pt] {6};
\path (b) -- (j)  node [above right = 6pt] {6};
\path (c) -- (d)  node [above right = 6pt] {5};
\path (k) -- (m)  node [below left = 9pt] {5};
\path (v) -- (w)  node [below left = 9pt] {4};
\path (m) -- (p)  node [above right = 6pt] {4};
\path (p) -- (t)  node [below left = 9pt] {2};
\path (m) -- (s)  node [below left = 9pt] {6};
\path (m) -- (q)  node [above right = 6pt] {6};
\path (u) -- (x)  node [above right = 6pt] {2};
\path (v) -- (u)  node [below left = 9pt] {5};
\path (t) -- (w)  node [above right = 6pt] {2};
\path (l) -- (p)  node [below left  = 9pt] {5};
\end{scope}

\end{tikzpicture}
}

\medskip
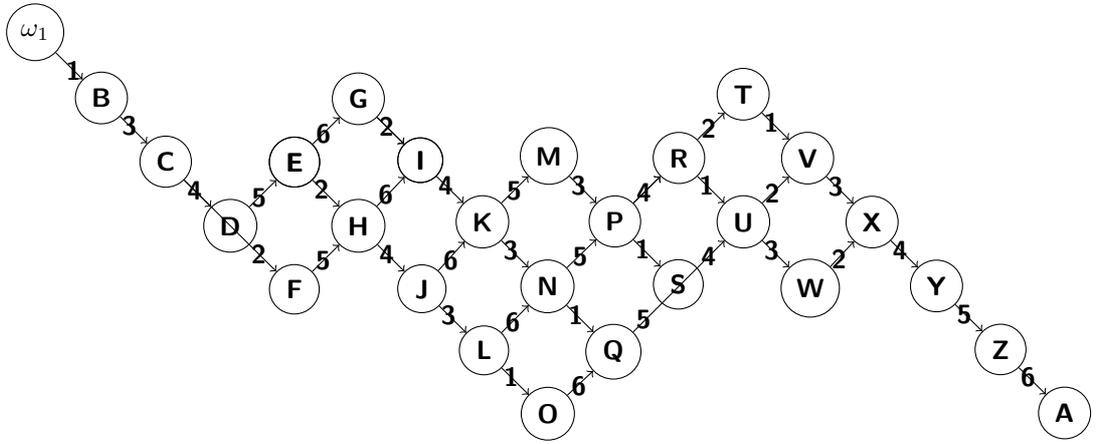
Figure 1. Weight diagram for $V_{\om_1}$ for $E_6$

\end{center}

\bl{e60} Let $G=E_6$ and let $\mu_1,\mu_2,\mu_3,\mu_4$ be distinct weights of $V=V_{\om_1}$. Then $\mu_i-\al,\mu_j-\al\in\Om(V)$ for some root $\al$ and   $i,j\in\{1,2,3,4\}$, $i\neq j$.  \el

\bp We use the weight diagram of $V$ given on Figure 1 above. (This is a more precise version of a diagram provided in \cite{PSV}.) The weights of $V$ are shown as the nodes denoted   by capital letters $A\ld Z$ and $\om_1$,  and two nodes are linked with an arrow with label $i$ if the second one is obtained from the first one by substracting   simple weight $\al_i$. For instance, $F-\al_5=H$, $Z-\al_6=A$. 

All  weights of $V_{\om_1}$ lie in the Weyl group orbit of $\om_1$, so we can assume that $\mu_1=H$. Suppose the contrary. 

Let $\Phi=\Phi(G)$. For a weight $\lam$ of $V$ set $[\lam,\Phi]= \{\nu\in \Om(V): \lam-\al,\nu-\al\in \Om(V)$ for some $\al\in\Phi\cup 0\}$. In addition, for $\al\in \Phi$ we set $[\lam,\al]=\lam \cup \{\nu\in \Om(V):\lam-\al,\nu-\al\in \Om(V)\}$. So $[\lam,\Phi]=\cup_{\al\in \Phi} [\lam,\al]$.
Note that $\lam\in[\lam,\Phi]$.

One can easily observe, using the weight diagram, that $[H,\Phi]$ consists of 17 weights,
and the remaining weights are
$\Delta_1=\{D,G,K,N,Q,R,U,W,Y,A\}$. (It is less obvious that $B,X\in [H,\Phi]$; however, 
$B\in[H,\be]$, where $\be=\al_3+\al_4+\al_5+\al_6 \in\Phi$. Similarly, $Z\in[H,-\gamma]$, where $\gamma=-\al_3-\al_4-\al_5-\al_2 \in\Phi$.) 

\itf $\mu_2,\mu_3,\mu_4\in\Delta_1$. We show that $\mu_i-\al,\mu_j-\al\in\Om(V)$ for some root $\al$ and   $i,j\in\{2,3,4\}$, $i\neq j$. 

Suppose first that $R,Q\in\{\mu_2,\mu_3,\mu_4\}$. Then we can assume that $R=\mu_2,Q=\mu_3$. Then $\Delta_1\subset [R,\Phi]\cup [Q,\Phi]$, so the result follows. 

So we assume that either $R$ or $Q$ are not in $\{\mu_2,\mu_3,\mu_4\}$. 
Observe that 
$D\succ G\succ K\succ N\succ R\succ U\succ W\succ Y\succ A$ and $D\succ G\succ K\succ N\succ Q\succ U\succ W\succ Y\succ A$.  Therefore, we can assume that $\mu_2\succ \mu_3\succ \mu_4$. 

 Suppose that $\mu_4=A$. As $G,K,N,R,U,W,Y\in[A,\Phi]$, 
we have $\mu_2=D,\mu_3=Q$, a contradiction as $D\in [Q,\al_5]$. 

 Suppose that $\mu_4=Y$. As $D,K,N,Q,W\in[Y,\Phi]$, we observe that $\mu_2,\mu_3\in\{G,R,U,W\}$. However, $G-\al_2,R-\al_2,U-\al_2,W-\al_2\in\Om(V)$, which yields a contradiction.

 Suppose that $\mu_4=W$. As $D,G,N,Q,R,U\in[W,\Phi]$, we conclude that $\mu_2,\mu_3\in\{G,K,R,U\}$. Since $G-\al_2,R-\al_2,U-\al_2\in\Om(V)$, we have  a contradiction. 

  Suppose that $\mu_4=U$. As $D,K,R,G\in[U,\Phi]$, 
we get a contradiction. 

 Suppose that $\mu_4=R$. We have  $D,G,K,N\in [R,\Phi]$ so $\mu_2=D,\mu_3=Q$,  a contradiction. 
 
Suppose that $\mu_4=Q$. We have $D,G,K,N\in [Q,\Phi]$ 
so $\mu_2=D,\mu_3=R$,  a contradiction. 

 Finally, suppose that $\mu_4\in\{K,N\}$. Then $[K,\Phi]$ and $[N,\Phi]$ contain $D,G,K,N$; this yields  a contradiction. 
\enp






\begin{lemma}\label{e62} Let $G=E_6$, $V=V_{\om_1}$ and let $M$ be an  \ir $G$-module. Let $s\in T_{reg}$. Suppose that s is almost cyclic  on $M\otimes V$. Then $m_s(M\otimes V)\leq3$. \el

\begin{proof} Suppose the contrary, and let $e$ be an \ei of $s$ on $V$ of \mult at least 4.  

 (i) Suppose first that $M$ is trivial. Let $\mu_i$,  $i=1,2,3,4$,  be distinct weights of $V$ such that $\mu_1(s)=\mu_i(s)$ for every $i$. By Lemma \ref{e60},  there are distinct $i,j\in\{1,2,3,4\}$ and 
$\al\in\Phi(G)$ such that $\mu_i-\al,\mu_j-\al\in\Om(V_{\om_1})$. This contradicts Lemma \ref{hh7}.

(ii) Suppose that $M$ is not trivial. By Lemma \ref{ma2}(3), $s$ is cyclic on $M$ and $V$.
Then $e=a_ib_i$ for $i=1,2,3,4$,
where $a_1,a_2,a_3, a_4$ are distinct \eis of $s$ on $V$ and $b_1,b_2,b_3, b_4$ are distinct \eis of $s$ on $M$.
We have  $a_i=\mu_i(s)$ for some distinct weights $\mu_1,\mu_2,\mu_3, \mu_4$ of $V$. As above, there is a root $\al\in\Phi$ and $i,j\in\{1,2,3,4\}$, $i\neq j$,  such that $\mu_i-\al,\mu_j-\al\in\Omega(V_{\om_1})$. Then $c_i=(\mu_i-\al)(s)$ and $c_j:= (\mu_j-\al)(s)$ are \eis of $s$ on $V=V_{\om_1}$, which are distinct as $s$ is cyclic on $V$. In addition, $c_i\neq a_i$, $c_j\neq a_j$
as otherwise $\al(s)=1$, which  is false as $s$ is regular.  So $c_ib_i=c_jb_j\neq a_ib_i=a_jb_j$.
This is a contradiction as $s$ is almost cyclic  on $M\otimes V$.\end{proof}

The \f example shows that the bound in Lemma \ref{e62} is attained.

\begin{examp} {\rm Let $G=E_6$.  Then $G$ contains a subgroup $H$, say, of type $F_4$. Let $\eta_1,\eta_2,\eta_3,\eta_4$ be the fundamental weights of $F_4$. 
Let $V$ be an \ir $G$-module of \hw $\om_1$. It is well known that $V|_H$ has a composition factor $U$ of \hw $\eta_4$ whose dimension is $\dim V-1$ or $\dim V-2$,
the latter for $p=3$.  Then the other composition factors of $V_{\om_1}|_{H }$ are trivial. Let $s\in F_4$ be a regular semisimple element in general position, that is, the cyclic group $\lan s \ran$ is dense in a maximal torus $S$, say, of $F_4$. Then $s$  is regular and separates the weights of $U$, and hence $s$ is almost cyclic on $V$. 
In addition, $\dim U^s=\dim V^s=2$ or 1, the latter for $p=3$. \itf $\dim V^s=3$. }
\end{examp}

\begin{lemma}\label{e65} Let $G=E_6$, 
let $V$ be an  \ir $G$-module and $s\in T_{reg}$. Suppose that s is almost cyclic  on $ V$. Then $m_s( V)\leq3$ unless the highest weight of $V$ is $p^t\om_2$. \el

\begin{proof} Let $\om$ be the \hw of $V$. Recall that the non-zero weights of $V$ must be of \mult 1 (Theorem  \ref{ag8})
and $s$ is regular by Theorem \ref{c99}.  Therefore, either  
 $\om=p^t\om'$, where $\om'\in\{\om_1,\om_6,\om_2\}$ by \cite[Theorem 2]{TZ2} or $V=V'\otimes M$, where $M$ is a non-trivial \ir $G$-module and the highest weight of $V'$ is $p^t\om_1$ or $p^t\om_6$ (here $\om'\neq \om_2$ by Lemma \ref{ma2}(3)). Note that the \ir \reps with highest weights $p^t\om_1$ and $p^t\om_6$ are dual to each other. So the result follows from Lemma \ref{e62}.
(If $\om'=\om_2$ then $m_s( V)=6$ by Theorem \ref{do1}(1) and   the fact that $V_{\om_2}$ is the adjoint $G$-module.)   
\enp

\bl{ee7}
Let $G=E_7$,  $V=V_{\om_7}$ and $s\in T_{reg}$. Suppose that $s$ is almost cyclic  on $ V$. Then $m_s( V)\leq 2$. \el

  \begin{proof}  
  Let $H$ be a subgroup of $G$ isomorphic to a simple \ag of type $A_7$,  so the maximal tori of $H$ are  maximal  in $G$. Hence we can assume that $T\subset H$. 
Let $\eta_1\ld \eta_7$ be the fundamental weights of $H$. Then the restriction $V|_{H}$ is $V_{\eta_2}+V_{\eta_6}$. The result follows from the \f lemma:
 \end{proof}

 \bl{aa8}
  Let H be a simple \ag of type $A_{n-1}$, $n\geq 4$. Let $U:=V_{\eta_2}\oplus V_{\eta_{n-2}}$. Let $s\in T_{reg}$. Suppose that $s$ is almost cyclic  on U. Then  $m_s(U)\leq 2$. \el

 \begin{proof}
 We can assume that  $g=\diag(a_1\ld ,a_n)\in SL_n(F)$, where $a_1\ld a_n\in F$ and $a_1\cdots a_n =1$.
 Then the \eis of $s$ on $U$ are $a_ia_j,a_k\up a_l\up $ for  $1\leq i,j,k,l\leq n$ and $i<j,k<l$.
  Suppose that  $s$ is not cyclic on $U$. Then either $a_ia_j=a_ka_l$ for $(i,j)\neq (k,l)$ or $a_ia_j=a_k\up a_l\up$ for some $i,j,k,l$.
  Then $(a_ia_j)\up $ is an \ei of $s$ of \mult at least 2, so $a_ia_j=(a_ia_j)\up$ as $s$ is almost cyclic. Then $a_i^2=a_j^{-2}$ and   $a_j=\theta a_i\up$ with $\theta=1$ or $-1$. \itf $\theta$ is an \ei of $s$ on $U$ of \mult at least 2.
  We show that the \mult of $\theta$ is exactly 2. Suppose that contrary. Then $\theta=a_ka_l$ for some $k<l$ with $(k,l)\neq (i,j)$.  Then $a_ja_l=a_i\up a_k\up$ and these are \eis of $s$ on $U$. As  $s$ is almost cyclic,  we must have $a_i\up a_k\up=a_ja_l=\theta=a_ka_l=a_k\up a_l\up $,  whence $a_j=a_k, a_i=a_l$. As $s$ is regular, we  have $j=k, i=l$. 
This is  a contradiction.  
 \end{proof}

\bl{77} Let $G=E_7$, 
let $V$ be an  \ir $G$-module and $s\in T_{reg}$. Suppose that s is almost cyclic  on $ V$. Then $m_s( V)\leq2$ unless the highest weight of $V$ is $p^t\om_1$. 
\el

\bp  Let $\om$ be the \hw of $V$. The non-zero weights of $V$ are of \mult 1 (Theorem  \ref{ag8}) and $s$ is regular by Theorem \ref{c99}.  Therefore, either  
 $\om\in\{p^t\om_7,p^t\om_1\}$ by \cite[Theorem 2]{TZ2} or $V=V'\otimes M$, where $M$ is a non-trivial \ir $G$-module 
and the highest weight of $V'$ is $p^t\om_7$  (this cannot be $p^t\om_1$ by Lemma \ref{ma2}(3)). So the result follows from Lemma \ref{ee7}.
(If $\om=p^t\om_1$ then $m_s( V)=7$ by Theorem \ref{do1}(1) and   the fact that $V_{\om_1}$ is the adjoint $G$-module.)   
\enp

\bl{ff44} Let $G=F_4$, $p=3$,   let $V $ be a p-restricted \ir $G$-module with \hw $\om_4$ and $s\in T_{reg}$. Suppose that s is almost cyclic  on $ V$. Then 
$m_s( V)\leq 2$. \el

\bp In this case $\dim V^T=1$ and hence $V^s=V^T$ by Theorem \ref{do1}. Therefore, if $m_s( V)>1$ and $e$ is the \ei of $s$ on $V$ of \mult $m_s( V)$ then $e=-1$
by Lemma  \ref{tt4}(2). 
 
The non-zero weights of $V=V_{\om_4}$ are short roots, which form a single orbit of the Weyl group of $G$. So we can assume that $\om_4(s)=-1=(-\om_4)(s)$.
Let  $\mu\succ 0$ for $\mu\in\Om(V)$. One can check that $\om_4-\mu$ is a root.  
If  $m_s( V)> 2$ then $\mu(s)=-1=(-\mu)(s)$ for some $0\neq \mu\in\Om(V)$, $\mu\neq \pm \om_4$. So we can assume that  $\mu\succ 0$. Then $\om_4-\mu$
is a root and $(\om_4-\mu)(s)=1$, This is a contradiction as $s$ is regular.\enp

 The case of $G=G_2$ is essentially treated in \cite{TZ22}:  

\begin{lemma}\label{2g3} {\rm \cite[Proposition 4.9]{TZ22}} Let $G=G_2$, $V_i=V_{\om_i}$ for $i=1,2$ and let $s\in T_{reg}$.

$(1)$  Suppose that s is not cyclic on $V_1$. Then $p\neq 2$  and the \eis of s   on $V_1$ are $\{1,-1,-1,-b,b\up, b,-b\up\}$,  where $b\in F,b^2\neq 1$. In particular, s is  almost cyclic  on $V_1$ \ii $b^2\neq  -1$.
 In addition, $\{\al(s): \al$ a long root$\} =\{b,b\up,-b,-b\up, b^2,b^{-2}\}$.

$(2)$  Let $p=3$.  Then  s is almost cyclic  on $V_2$ and cyclic   on $V_1$. More precisely, if
s is not cyclic on $V_2$ then the \eis of s   on $V_2$ are
   $\{1,-1,-1, a^3,-a^3, a^{-3},-a^{-3}\}$ for some $a\in F, a^{4}\neq 1$, and those on $V_1$ are  $\{1,a,a\up, -a,-a\up, -a^{2},-a^{-2}\}$. \el
 
\begin{corol}\label{gk6} Theorem {\rm  \ref{ab9}}  is true for G of type $G_2$.  
\end{corol}

\bp Suppose first that $V$ is $p$-restricted. Then, by Theorem \ref{ag8},   we need to deal with $FG$-modules $V$ of \hw $\om\in\{\om_1,\om_2\}$.
The case $\om=\om_1$ is treated in Lemma \ref{2g3}(1). Let $\om=\om_2$. The case with $p=3$ is examined in Lemma \ref{2g3}(2).
Let $p\neq 3$. Then $\dim V^T=2$ and  hence $\dim V^s\geq 2$. If $\dim V^s> 2$ then $s$ is not regular as the non-zero weights of $V$ are roots.
This contradicts Theorem \ref{c99}.   

Consider then the case where $V$ is not $p$-restricted. Then the result follows from Lemma \ref{ma2}(4).\enp

\subsection{Groups  of type $A_n$}

Let $G$ be a simple \ag of type $A_n$ and  $\om_1\ld \om_n$  the fundamental weights of $G$. As above, for a dominant weight $\om$ we denote by $V_\om$
an \ir $G$-module with \hw $\om$.

\subsection{Tensor-indecomposable modules}

\bl{ep9} Let $V=V_{\om_i}$, $i\leq (n+1)/2$, and $n+1<(k+1)i$ for some integer $k>0$. Suppose that $s\in T_{reg}$ and s is almost cyclic on V. Then $m_s(V)\leq k$. In particular, $m_s(V)\leq (n+1)/2$ for $i>1$. 
\el

\bp The weights of $V$ are $\ep_{j_1}+\cdots +\ep_{j_i}$ for $1\leq j_1<...<j_i\leq n+1$. If $\lam=\ep_{j_1}+\cdots +\ep_{j_i}$ then we set ${\rm supp}(\lam)=\{j_1\ld j_i\}$. Suppose the contrary, and let
$\mu_1\ld \mu_{k+1}$ be distinct weights of $ V$ such that $\mu_1(s)=\cdots = \mu_{k+1}(s)$.  As $n+1<(k+1)i$, it follows that there is $d\in \{1\ld n+1\}$ and a pair $\mu_r,\mu_t$ with $1\leq r<t\leq k+1$ such that $d\in {\rm supp}(\mu_r)\cap {\rm supp}(\mu_t)$.  
In addition, $|{\rm supp}(\mu_r)\cup {\rm supp}(\mu_t)|\leq2i-1$ and $2i-1\leq n$ by assumption. So $l\notin {\rm supp}(\mu_r)\cup {\rm supp}(\mu_t)$ for some $l\in\{1\ld n+1\}$. Set $\al=\ep_d-\ep_l$. \itf  $\mu_r-\al,\mu_t-\al$ are weights of $V$. This contradicts Lemma \ref{hh7}.

For the additional claim, take $k=(n+1)/2$. Then $(k+1)i>n+1$ as $i>1$. So the claim follows from the above.
\enp

\bl{an4}
 Let $G=A_{n},n>2$, $M=V_{\om_i}$ with $2\leq i\leq (n+1)/2$, $M_1=V_{\om_1}$, $M_2=V_{\om_{i-1}}$. Let $\lam,\lam'\in \Om(M_2)$, $\lam\neq\lam'$. Then there are distinct weights $\mu,\mu'\in \Om(M_1)$ such that $\lam+\mu,\lam+\mu',\lam'+\mu,\lam'+\mu'\in\Om(M)$. In addition, $\Om(M)\subset \Om(M_1)+\Om(M_2)$. \el

\begin{proof}  The weights of $M_1$ are $\ep_1\ld \ep_{n+1}$ and those of $M_2$ are $\ep_{j_1}+\cdots+\ep_{j_{i-1}}$ for $1\leq j_1<\cdots <j_{i-1}\leq n+1$.
We can assume (using the Weyl group) that $\lam=\ep_1+\cdots+\ep_{i-1}$ and $\lam'=\ep_{j_1}+\cdots+\ep_{j_{i-1}}$ with $j_1<\cdots<j_{i-1}\leq 2i-2\leq n-1$. Then we take $\mu=\ep_{n},\mu'=\ep_{n+1}$.

For the additional claim, we have $\Om(V_{\om_1})+\Om(V_{\om_{i-1}})=\Om(V_{\om_1+\om_{i-1}})\supset \Om(V_{\om_i})$ as $\om_{i}\prec\om_1+\om_{i-1}$, see Lemma \ref{wtlattice}. Whence the result. \end{proof}

\bl{aan}
Let $G=A_{n},n>2,$ and let $s\in  T_{reg}$.

$(1)$ If s is almost cyclic  on $V_{\om_i}$ for $2\leq  i\leq (n+1)/2$,
 then s is cyclic on $V_{\om_j}$ for every $j<i$.

 $(2)$ If s is almost cyclic  on $V_{a\om_1}$ for $1<a\leq p-1$, then s is cyclic on  $V_{b\om_1}$ for $1\leq b< a$. \el

\begin{proof} Set $M_1=V_{\om_1}$, and in case (1) set $M=V_{\om_i}$,   $M_2=V_{\om_{i-1}}$, in case (2) set $M=V_{a\om_1}$, $M_2=V_{(a-1)\om_1}$.

$(i)$ $s$ separates the weights of $M_2$.

In case (2) this follows from Lemma \ref{wtlattice}(2) and Lemma \ref{td2}(1).

In case (1) suppose the contrary. Then there are distinct weights $\lam_1,\lam_2$  of $M_2$ such that $\lam_1(s)=\lam_2(s)$.
By Lemma \ref{an4}, there are distinct weights  $\mu_1,\mu_2$ of $M_1$ such that $\lam_k+\mu_l$ are weights of $M$ for $k,l\in\{1,2\}$.   Then $(\lam_1+\mu_1)(s)=(\lam_2+\mu_1)(s)$ and $(\lam_1+\mu_2)(s)=(\lam_2+\mu_2)(s)$. As $s$ is almost cyclic  on $M$, it follows that $(\lam_1+\mu_1)(s)=(\lam_1+\mu_2)(s)$,  and then $\mu_1(s)=\mu_2(s)$. This is a contradiction, as every regular semisimple element is cyclic on  $V_{\om_1}=M_1$.

$(ii)$ $s$ is cyclic on $M_2$.

Indeed, the weights of $V_{\om_j}$ and of $V_{b\om_1}$ are well known to be of \mult 1, so the claim follows from $(i)$.

$(iii)$ The result follows by induction on $i$ in case (1) and on $a$ in case (2).\end{proof}

We use Lemma \ref{an4} to prove

\bl{a55} Let $G=A_{n}$, $n>1$, $M_1=V_{\om_1}$, $M=V_{a\om_i+(p-1-a)\om_{i+1}}$ with $1\leq i\leq (n+1)/2$,  $M_2=V_{\om_{i-1}+(a-1)\om_i+(p-1-a)\om_{i+1}}$ with $1\leq a\leq p-1$. (If $i=1$ then $\om_{i-1}$ is interpreted as the zero weight.)

$(1)$ Let $\lam,\lam'\in \Om(M_2)$, $\lam\neq\lam'$. Then there are distinct weights $\mu,\mu'\in \Om(M_1)$ such that $\lam+\mu,\lam+\mu',\lam'+\mu,\lam'+\mu'\in\Om(M)$. In addition, $\Om(M)\subseteq\Om(M_1)+\Om(M_{2})$.

$(2)$ If s is almost cyclic on M then s separates the weights of $M_2$.  \el

\begin{proof} (1) We have $\Om(M)=\Om(V_{\om_{i}})+\Om(V_{(a-1)\om_i+(p-1-a)\om_{i+1}})$ by Lemma \ref{wtlattice}, and $\Om(V_{\om_i})\subset \Om(V_{\om_{1}})+\Om(V_{\om_{i-1}})$ by Lemma \ref{an4}. So $$\Om(M)\subset \Om(M_1)+\Om(V_{\om_{i-1}})+\Om(V_{(a-1)\om_i+(p-1-a)\om_{i+1}})=\Om(M_1)+\Om(M_{2})$$ as $\Om(M_2)=\Om(V_{\om_{i-1}})+\Om(V_{(a-1)\om_i+(p-1-a)\om_{i+1}})$ again by Lemma \ref{wtlattice}.

Let $i>1$. Then we can write $\lam=\beta+\gamma$, $\lam'=\beta'+\gamma'$ with $\beta,\beta'\in \Om(V_{\om_{i-1}})$, $\gamma,\gamma'\in \Om(V_{(a-1)\om_i+(p-1-a)\om_{i+1}})$. By Lemma \ref{an4}, there are distinct 
weights $\mu,\mu'\in\Om(V_{\om_1})$ such that $\mu+\beta,\mu'+\beta,\mu+\beta',\mu'+\beta'\in\Om(V_{\om_i})$. As $\Om(V_{\om_i})+\Om(V_{(a-1)\om_i+(p-1-a)\om_{i+1}})=\Om(V_{a\om_i+(p-1-a)\om_{i+1}})$,
we have $\mu+\beta+\gamma,\mu'+\beta+\gamma,\mu+\beta'+\gamma',\mu'+\beta'+\gamma'\in\Om(M)$, and the result follows.

Let $i=1$ so $V_{\om_{i-1}}$ is   the trivial module. Then $\Om(M)=\Om(M_1)+\Om(M_{2})$, and
 the claim is obvious.

(2) We mimic the proof of Lemma \ref{aan}. Suppose the contrary, and let
$\lam_1,\lam_2$  be distinct weights of   $M_2$ such that $\lam_1(s)=\lam_2(s)$. By the above, there are distinct weights $\mu_1,\mu_2$   of $M_1$ such that $\lam_j+\mu_k$ are weights of $M$ for $i,k\in\{1,2\}$. Then $(\lam_1+\mu_j)(s)=(\lam_2+\mu_j)(s)$ for $j=1,2$. As $s$ is almost cyclic  on $M$, it follows that $(\lam_1+\mu_1)(s)=(\lam_1+\mu_2)(s)$,  and then $\mu_1(s)=\mu_2(s)$. This is a contradiction, as $s$ is  regular.  \end{proof}

\begin{propo}\label{pr5} For $G=A_{n}$  let $M\in\Irr(G)$ be  p-restricted and
 $s\in T_{reg}$. Suppose that s is almost cyclic on M. Then $m_s(M)\leq n+1$.\end{propo}

\begin{proof} Suppose first that the weight zero multiplicity of $M$ is greater than $1$. Suppose the contrary; as $s$ is almost cyclic, only the \ei 1 \mult is greater than 1. Then the result follows from Theorem \ref{do1}(1) and Table 1.

Suppose that the weight zero multiplicity of $M$ is at most $1$. By Theorem \ref{ag8},  all weight multiplicities of $M$ equal $1$. By Table 2, the highest weight $\om$ of $M$ is in the following list: $0,\om_i$ $(1\leq i\leq n)$, $a\om_1,$  $a\om_n$, $a\om_i+(p-1-a)\om_{i+1}$, where $1\leq i\leq n-1$ and $1\leq a<p$.

 Note that we assume  $\om\neq 0,\om_1,\om_{n}$ as these cases are trivial.  
   Since the list of the \ei multiplicities of $M$ is the same as for the dual of $M$, it suffices to deal with $\om_i,a\om_1,a\om_i+(p-1-a)\om_{i+1}$ for $i\leq (n+1)/2$.

Suppose the contrary.  Let $M_1=V_{\om_1}$ and

$$M_2=\begin{cases} V_{\om_{i-1}} &if \,\,\om=\om_i, i>1;\cr
V_{(a-1)\om_{1}}&if \,\, \om=a\om_1, a>1;\cr
V_{\om_{i-1}+(a-1)\om_i+(p-1-a)\om_{i+1}}& if \,\, \om=a\om_i+(p-1-a)\om_{i+1}. \end{cases} $$
(If $i=1$ then $\om_{i-1}$ is understood to be the zero weight.)
By Lemmas \ref{aan} and \ref{a55}(2), $s$ separates the weights of $M_2$.

 Note that $M$ is a constituent of $M_1\otimes M_2$. Indeed, $V_{\om_{i}}$ is a constituent of $V_{\om_1}\otimes V_{\om_{i-1}}$
 and  $M_2$ is a constituent of $V_{\om_{i-1}}\otimes V_{(a-1)\om_i+(p-1-a)\om_{i+1}}$, whence the claim.

 Furthermore,  $\Om(M)\subseteq \Om(M_1)+\Om(M_2)$ (see Lemma \ref{a55}). As all \ei multiplicities of
 $M$ equal 1, there are $n+2$ distinct weights $\nu_1\ld \nu_{n+2}$ of $M$ such that $\nu_i(s)=e$. Then we can write
 $\nu_i=\lam_i+\mu_i$, where $\lam_i,\mu_i$ are weights of $M_1,M_2$, respectively, for $i=1\ld n+2$.
 Then $e=\lam_i(s)\mu_i(s)$.
 As $M_1=V_{\om_1}$ has exactly $n+1$ distinct weights, we have $\lam_i=\lam_j$ for some $i,j\in\{1\ld n+1\}$ with $i\neq j$.
 Then $\mu_i\neq \mu_j$ as otherwise $\nu_i=\lam_i+\mu_i=\lam_j+\mu_j=\lam_j+\mu_j=\nu_j$, which is false. However, $\lam_j(s)\mu_j(s)=e=\lam_i(s)\mu_i(s)=\lam_j(s)\mu_i(s)$, whence $\mu_i(s)=\mu_j(s)$.
 This is a contradiction as $s$ separates the weights of $M_2$. \end{proof}

The bound in Proposition \ref{pr5} is not probably sharp in general, at least, for $G=A_n$ we have no example of a $p$-restricted $G$-module $M$ with $m_s(M)>n$ for $s$ regular and almost cyclic on $M$. Moreover, the equality   $m_s(M)=n$ probably holds only for $M$ adjoint, and we have no sharp bound for other $p$-restricted modules. 
 In Lemmas \ref{ep9} and \ref{we4} we consider some special cases.
  Example \ref{ex5} shows that, for $G=A_n$, $n>3$ odd, there exists a regular semisimple element $s\in G$
and a $p$-restricted \ir $G$-module $V$ which is not a twist of the natural or adjoint $G$-module such that $s$ is almost cyclic on $V$ and $m_s(V)=(n+1)/2$.  (If $V$ is not a twist of a $p$-restricted $G$-module  then there are examples with $m_s(V)=n+1$, see Example \ref{e2x}.)

\bl{we4} Let $G=A_n=SL_{n+1}(F)$ and $p\neq 2$.  If $s\in T_{reg}$  then $m_s(V_{2\om_1})\leq (n+3)/2$. 
\el

\bp By \cite[Theorem 1.2]{TZ22}, if $s$ is almost cycic on $V_{2\om_1}$ then $s$ is strictly regular.  Suppose that  $\mu_1\ld \mu_l\in\Om(V_{2\om_1})$ be distinct weights such that $\mu_1(s)=\cdots =\mu_l(s)=e$. The weights of $V$ are $2\ep_i$ ($i=1\ld n+1$) and $\ep_i+\ep_j$ $(i,j\in\{1\ld n+1\}, i<j)$; this can be recorded as 
$\ep_i+\ep_j$ $(i,j\in\{1\ld n+1\}, i\leq j)$.
Suppose first that  $\mu_i ,\mu_j,\mu_k\in \{2\ep_1\ld 2\ep_{n+1}\}$ for some distinct $i,j,k\in\{1\ld l\}$. By reordering of $\mu_1\ld \mu_l$ we can assume that $\mu_1,\mu_2,\mu_3\in \{2\ep_1\ld 2\ep_{n+1}\}$, 
and let $\mu_i=2\ep_{r(i)}$ for $i=1,2,3$. Then  for $i,j\in\{1,2,3\}$ we have $1=(\mu_i-\mu_j)(s)=(2\ep_{r(i)}-2\ep_{r(j)})(s)=\ep_{r(i)}(s)^2\ep_{r(j)}(s)^{-2}=(\ep_{r(i)}(s)\ep_{r(j)})(s)^{-1})^2, $ whence $\ep_{r(i)}(s)\ep_{r(j)}(s)^{-1}=\pm1$. 
In fact, $\ep_{r(i)}(s)\ep_{r(j)}(s)^{-1}=-1$, as $\ep_{r(i)}(s)=\ep_{r(j)}(s)$ implies $s\notin T_{reg}$. So $\ep_{r(i)}(s)=-\ep_{r(j)}(s)$. However, this cannot happen for three choices $i,j\in\{1,2,3\}, i\neq j$. Therefore, at most two
weights $\mu_i$ $(1\leq i\leq l)$ are in the set  $2\ep_1\ld 2\ep_{n+1}$. 

Note that $\Om(V_{\om_2})=\Om(V_{2\om_1}) \setminus \{2\ep_1\ld 2\ep_n\}$.
Let $l'$ be the number of weights $\mu_1\ld \mu_l$ that lie in $\Om(V_{\om_2})$. 
By the above  $l'\geq l-2$.    
By Lemma \ref{ep9}, we have $l'\leq m_s(V_{\om_2})\leq (n+1)/2$. If $l'=l$ or $l-1$ 
then the result follows. Suppose that $l'=l-2$, and we can assume that $\mu_1\ld \mu_{l-2}\in \Om(V_{\om_2})$ and that $\mu_{l-1}=2\ep_n$, $\mu_{l}=2\ep_{n+1}$. 
Then $\ep_n,\ep_{n+1}$ do not lie in  the support of any $\mu_t$ with $t=1\ld l-2$.
Indeed, otherwise $(\ep_i+\ep_k)(s)=2\ep_k(s)$ for $k\in\{n,n+1\}$ and some $i$,
whence $(\ep_i-\ep_k)(s)=1$,  which is false as $s$ is regular. 
Therefore, we can apply  Lemma \ref{ep9} to $G=A_{n-2}$ in place of $A_n$
to conclude that $l-2\leq (n-1)/2$, and hence $l\leq (n+3)/2$, as required. 
(Alternatively, we can obtain this conclusion from the proof of Lemma \ref{ep9} as 
there are at most $(n-1)/2$ weights $\mu_1\ld \mu_l$ whose supports are pairwise disjoint.) 
\enp

\begin{examp}\label{ex5} {\rm 
$(1)$ Let $G=SL_{2m}(F)$ and $s=\diag(a_1\ld a_m,a_m\up\ld a_1\up)\in G$, where the elements $1,a_1\ld a_m\in F$ are algebraically independent and non-zero. Let $\om=\om_2$. Then $V_\om^T=0$, $s$ is almost cyclic on $V_\om$ and $\dim V_\om^s=m=(n+1)/2$, where $n=2m-1$ is  the rank of $G$.

If $p\neq 2$ and $s=\diag(-a_1\ld -a_m,a_m\up\ld a_1\up)$ then $-1$ has \mult $m=(n+1)/2$ and the other \eis of $s$ on $V_{\om_2}$ are of \mult $1$.

$(2)$ Let $G=SL_{2m+1}(F)$, $p\neq 2$ and $s=\diag(1,a_1\ld a_m,a_m\up\ld a_1\up)\in G$, where the elements $a_1\ld a_m\in F$ are as above. Let $\om=2\om_1$. Then $V_\om^T=0$ and $\dim V_\om^s=m+1=(n+2)/2$, where $n=2m$ is the rank of $G$. In this case $s\in H\cong B_{m}$.

$(3)$ Let $G=SL_{2m}(F)$ and $s=\diag(b,-b,b^2a_1, a_1\up\ld b^2a_{m-1},a_{m-1}\up)\in G$, where the elements 
$a_1\ld a_{m-1}\in F$ are as above and $b^{n+1}=(-1)^{(n+1)/2}$, $b\in F$. Then the \ei $b^2$ has \mult $m+1=(n+3)/2$ on $V_{2\om_1}$ and the other \eis have \mult $1$.
(Here $n=2m-1$ is the rank of $G$.)}
\end{examp}

\subsection{Tensor-decomposable modules: special cases}

 Proposition \ref{pr5} states that Theorem \ref{ab9} is true for $p$-restricted \ir $A_n$-modules.
 We next prove it for tensor-decomposable modules.

\begin{lemma}\label{gc3}  {\rm \cite[Lemma 7]{TZ21}} Let $s\in T$ be a
non-central element. Let $\om$ be a dominant weight which is not $p$-restricted and
not of the form $p^k\mu$ for $\mu$ a $p$-restricted dominant weight. If $s$ is
almost cyclic on $V_\om$  then all weights of $V_\om$ are of \mult $1$.\el


 
For a $G$-module $V$ and $s\in G$ denote by $\deg_V(s)$ the degree of the minimum \po of $\rho(s)$,
where $\rho$ is the \rep afforded by $V$. If $s$ is semisimple then $\deg_V(s)=|E_V(s)|$
 is the number of distinct \eis of $s$ on $V$. Note that $s$ is cyclic on $V$ \ii $\deg_V(s)=\dim V$.

For $A,B\subset  F^\times$ we set $AB=\{ab: a\in A,b\in B\}$.

\bl{bb1} Let $M,M_1,M_2,V$ be non-trivial G-modules and $s\in T$.
Suppose that s is not scalar on both $M_1 $, $M_2$,  $E_M(s)=E_{M_1}(s)E_{M_2}(s)$ and
 $s$ is almost cyclic  on $M\otimes V$.

$(1)$   $\deg_{M_i\otimes V}(s)=\deg_{M_i}(s)\cdot \dim V$ for   $i=1,2$.

$(2)$ the \ei multiplicities of s on $M\otimes V$ do not exceed $d=\dim M_1$.\el

\begin{proof}  (1) We can assume $i=1$. By Lemma \ref{ma2}, $s$ is cyclic on $M$ and $V$. Then
 $\dim V=\deg_V(s)$. 
Suppose the contrary, that $\deg_{M_1\otimes V}(s)<\deg_{M_1}(s)\cdot \dim V$. Then there are distinct \eis $\mu,\mu'\in E_{M_1}(s)$  and $\nu, \nu'\in E_V(s)$  such that $\mu\nu=\mu'\nu'$. Then $\mu\lam\nu=\mu'\lam\nu'\in E_{M\otimes V}(s)$ for every  $\lam\in E_{M_2}(s)$. So $s$ is not almost cyclic  on $M\otimes V$ unless $s$  is scalar on $M_2$,   which is not the case by assumption. 

(2)  This follows from Lemma \ref{ma2}(1) as the \ei multiplicities of $s$ on $V$ equal 1. 
\end{proof}

To illustrate the use of Lemma \ref{bb1}, we state the following:

\begin{corol}\label{tpa} Let $G=A_n$, $\om=a\om_1$ or $a\om_1+(p-1-a)\om_2$, where $1<a<p$. Let $M=V_{\om}$ 
and let $V$ be a non-trivial \ir G-module. Let $s\in (T\setminus   Z(G))$. Suppose that s is almost cyclic  on $M\otimes V$. Then the \ei multiplicities of s on $M\otimes V$ do not exceed $n+1$.\end{corol}

\begin{proof} Set $M_1=V_{\om_1}$ and $M_2=V_{(a-1)\om_1}$ or $V_{(a-1)\om_1+(p-1-a)\om_2}$, respectively. Then $\Om(M)=\Om(M_1)+\Om(M_2)$ by Lemma \ref{wtlattice}. So the result follows from Lemma \ref{bb1}.\end{proof}

Note that Corollary \ref{tpa} is not implied by Proposition \ref{pr5}. In addition, here and  in Lemmas \ref{an2a}, \ref{r44}  below we do not assume $s$ to be regular as this  automatically holds due to Lemma \ref{ss2}(2) when $s\notin Z(G)$.

\bl{an2a} Let $M,M_1,M_2,V$ be non-trivial G-modules. Let $s\in T\setminus Z(G)$. Suppose that

$(i)$ for every $f,f'\in E_{M_2}(s)$ there exist distinct \eis $e,e' \in E_{M_1}
(s)$ such that $ef,ef',e'f,e'f' \in E_M(s);$ 

$(ii)$ $E_M(s)\subseteq E_{M_1}(s)\cdot E_{M_2}(s);$

$(iii)$ 
$s$ is almost cyclic  on $M\otimes V$.

\noindent Then  $\deg_{M_2\otimes V}(s)=\deg_{M_2}(s)\cdot \dim V$, and   $m_s(M\otimes V)\leq
\dim M_1$.\el

\begin{proof}  By Lemma \ref{ma2}, $s$ is cyclic  on $V$, so $\dim V=\deg_V(s)$.  To prove the first statement of  the lemma, suppose (the contrary), that $\deg_{M_2\otimes V}(s)<\deg_{M_2}(s)\cdot \dim V$. Then there are distinct \eis $f,f' \in E_{M_2}(s)$ and \eis $v,v'\in E_V(s)$  such that $fv=f'v'$.

By $(i)$, there are distinct \eis $e,e'\in E_{M_1}(s)$ such that $ef,ef',e'f,e'f' \in E_M(s),$  and hence
$efv = ef'v',e'fv = e'f'v' \in E_{M\otimes V}(s)$. So $a:=efv $ and $b:=e'fv $ are distinct \eis of $s$ on $M\otimes V$, each of \mult at least 2. 
As $s$ is almost cyclic  on $M\otimes V$,    we must have $a=b$,      whence  $e=e'$, a contradiction.

  Turn to the second statement.   Set $d=\dim M_1.$ Suppose the contrary, that the \mult of some \ei $u$ of $s$ on $M\otimes V$ is greater than  $d$. By Lemma \ref{ma2}(3), $s$ is cyclic on $M$ and $V$, so there are distinct \eis $w_1\ld w_{d+1}$ of $s$ on $M$ and distinct \eis $v_1\ld v_{d+1}$ of $s$ on $V$ such that $u=w_iv_i$ for $i=1\ld d+1$. Furthermore, by $(ii)$, $w_i=e_if_i $ for $i=1\ld d+1$ for some \eis $e_i\in E_{M_1}(s)$  and some \eis $f_i\in  E_{M_2}(s)$. Therefore, $e_i=e_j$ for some $i,j\in\{1\ld d+1\}$.
Then  $w_j=e_if_j$, where $f_i\neq f_j$ as $w_i\neq w_j$, and $u=e_if_iv_i=e_if_jv_j$, whence $f_iv_i=f_jv_j$.
 Then $\deg_{M_2\otimes V}(s)<\deg_{M_2}(s)\deg_{V}(s)$ contradicting the first statement of the lemma.  \end{proof}

\bl{r44} Let $G=A_{n}$ and $s\in (T\setminus  Z(G))$.
Let $M=V_{\om_i}$  for $1<i<n$ or $V_{a\om_i+(p-1-a)\om_{i+1}}$  for $1<i<n-1$,
and let V be a non-trivial G-module. Suppose that $M\otimes V$ is \ir and $s$ is almost cyclic  on $M\otimes V$. Then
$m_s(M\otimes V)\leq n+1$. \el

\begin{proof} We can assume (by replacing $M$ by its dual) that $i\leq (n+1)/2.$ Set $M_1=V_{\om_1}$.

 Suppose first that $M=V_{\om_i}$. Set $M_2=V_{\om_{i-1}}$. The result follows from Lemma \ref{an2a} as soon as we show that the assumptions $(i), (ii)$ of that lemma hold in this case (as $(iii)$ is in the assumption). By Lemma \ref{wtlattice}(2), $\Om(M)\subseteq \Om(M_1)+\Om(M_2)$,   whence $(ii)$.
 To prove $(i)$, let $f,f'\in E_{M_2}(s)$. Then there are weights $\lam,\lam'\in\Om(M_2)$ such that $\lam(s)=f,\lam'(s)=f'$.
 By Lemma \ref{an4}, there are weights $\mu,\mu'\in \Om(M_1)$ such that  $\lam+\mu,\lam+\mu',\lam'+\mu,\lam'+\mu'\in\Om(M)$. Set $e=\mu(s),e'=\mu'(s)$. To get $(i)$, we need $e\neq e'$. Suppose $e=e'$; then $(\lam+\mu)(s)=(\lam+\mu')(s)$ and $\lam+\mu\neq \lam+\mu'$. This is a contradiction as $s$ is cyclic on $M$ by Lemma \ref{ma2}.

Suppose that $M=V_{a\om_i+(p-1-a)\om_{i+1}}$. Then we set
$M_2=V_{\om_{i-1}+(a-1)\om_i+(p-1-a)\om_{i+1}}$.
The result follows from Lemma \ref{an2a} if  the assumptions $(i), (ii)$ of that lemma hold in this case (since $(iii)$ is in the assumption). 
Lemma \ref{a55} yields $(ii)$.

To prove $ (i)$, let $f,f'\in E_{M_2}(s)$. Then there are weights $\lam,\lam'\in\Om(M_2)$ such that $\lam(s)=f,\lam'(s)=f'$.
By Lemma \ref{a55}, there are weights $\mu,\mu'\in \Om(M_1)$ such that $\lam+\mu,\lam+\mu',\lam'+\mu,\lam'+\mu'\in\Om(M)$.
Set $e=\mu(s),e'=\mu'(s)$. To get $(i)$, we need $e\neq e'$. Suppose $e=e'$; then  $(\lam+\mu)(s)=(\lam+\mu')(s)$ and $\lam+\mu\neq \lam+\mu'$. This is a contradiction as $s$ is cyclic on $M$ by Lemma \ref{ma2}.\end{proof}

\begin{theo}\label{pt7}
 Let G be of type $A_{n}$ and N an \ir G-module. Let  $s\in T_{reg} $. Suppose that s is almost cyclic  on N. Then    $m_s(N)\leq n+1$.\end{theo}
 
 \begin{proof} Let $N=N_1\otimes\cdots\otimes N_k$, where $N_1\ld N_k$ are tensor-indecomposable non-trivial \ir $G$-modules. 
 Then every $N_i$ for $i=1\ld k$ is a Frobenius twist of an infinitesimally \ir  $G$-module. 
The problem is well known to be equivalent to the case where one of $N_i$ is infinitesimally \ir and we can assume that this is $N_1$ by reordering the terms. 

If $k=1$ then the result is contained in Proposition \ref{pr5}. Let $k>1$.
Then, by Lemma \ref{ma2}, $s$ is cyclic on  every  $N_i$. Therefore, all weights of $N_i$ are of \mult 1. By \cite[6.1]{GS}, $N_i$ is a Frobenius twist of an \ir $G$-module, whose highest weight belongs to the list  $\{\om_j, 1\leq j\leq n$, \smallskip $a\om_1,a\om_{n}\,\, (1<a<p)$, $a\om_i+(p-1-a)\om_{i+1}, \,\, i=1\ld n-1, 0\leq a<p \}$.   \smallskip Set $M=N_1$, $V=N_2\otimes\cdots\otimes N_k$, so $V_2$ is \ir and non-trivial, whereas the highest weight $\om$ of $N_1$ lies in the above list. Then the result follows from Corollary \ref{tpa} and Lemma \ref{r44}.\end{proof}

One can expect that the upper bound for $m_s(V)$ for $V$ tensor-decomposable must be significantly lower than for $V$ tensor-indecomposable. The \f example 
makes evident that this is not so. In fact, we see  that the bound in Theorem \ref{pt7} is sharp. 

\begin{examp}\label{e2x}
{\rm Let $F$ be an algebraically closed field  of characteristic $p>0$, $G=A_n$, $n\geq 2$, and $V$  an \ir $G$-module with \hw $\om_1+p\om_n$.
Then there exists $s\in G$ such that $s$ is almost cyclic on $V$ and $m_s(V)=n+1$ for some $s\in G$.

\med
Indeed, set $s=\diag(a,a^p\ld a^{p^n})\in SL_{n+1}(F)$, where $a\in F$ is a primitive $((p^{n+1}-1)/(p-1))$-root of unity.
Then the matrix of $s$ on $V$ is $s\otimes s^{-p}$ and the \eis of this matrix are $a^{p^i-p^{j+1}}$ with $0\leq i,j\leq n$. If $i=j+1$ then this equals 1, as well as if $i=0,j=n$ as then $a^{1-p^{n+1}}=1$. So 1 is an \ei of $s\otimes s^{-p}$ of \mult at least  $n+1$. In fact, this is exactly $n+1$ as  $a^{p^i-p^{j+1}}=1$ implies $p^i-p^{j+1}\equiv 0\pmod {p^n+p^{n-1}+\cdots +p+1}$ which is false unless $i=j+1$ or $i=0,j=n$.  
 We show that all other \eis are of \mult 1. Suppose the contrary. Then $a^{p^i-p^{j+1}}=a^{p^k-p^{l+1}}$ for some choice of $0\leq i,j,k,l\leq n$ and  $i\neq j+1$, $k\neq l+1$ and $(i,j)\neq (0,n)\neq (k,l)$. As above, we observe that $i\neq k,j\neq l$. We have $p^i-p^{j+1}-p^k+p^{l+1}\equiv 0\pmod {p^n+p^{n-1}+\cdots  +p+1}$. Let $m=\min(i,j+1,k,l+1)$. 
If $m>0$ then $p^{i-m}-p^{j+1-m}-p^{k-m}+p^{l+1-m}\equiv 0\pmod {p^n+p^{n-1}+\cdots +p+1}$ and now $d:=\max(i-m,j+1-m,k-m,l+1-m)$ does not exceed $n$. If $d<n$ then the absolute value of the left hand side is at most $2p^{n-1}$ and the congruence does not hold. (If $p=2$ then $2^n+2^{n-1}+\cdots +1=2^{n+1}-1$ and the congruence fails too.) Suppose that $d=n$. Then $m=1$ and $\max(j,l)=n$. We can assume $l=n$ and then $j<n$. Then the  absolute value of $p^{i-m}-p^{j+1-m}-p^{k-m}+p^{n}$ does not exceed $p^n+p^{n-1}<p^n+p^{n-1}+\cdots +p+1$, hence the congruence above cannot hold. 

So $m=0$. As above, the  congruence fails unless possibly $d\in\{n,n+1\}$. 

Suppose that $d=n$. Then the  absolute value of $p^i-p^{j+1}-p^k+p^{l+1}$ does not exceed $bp^n-p^c-1$, where $b\in\{0,2\}, 0\leq c\leq n$. Then the congruence cannot hold.  

Suppose that $d=n+1$. Then we can replace $p^d$ by 1, and then 
the  absolute value of $p^i-p^{j+1}-p^k+p^{l+1}$ does not exceed $2+p^{n}+p^{n-1}$. If $n>2$ then 
  the congruence above cannot hold. Let $n=2$. We can assume that $l=2$ and then  $s\otimes s^{-p}=\diag(1,1,1, a^{1-p}, a^{2+p}, a^{1+2p},a^{p-1}, a^{-1-2p},a^{-2-p})$. One checks straightforwardly  that  this matrix is almost cyclic. 
} \end{examp}

\begin{proof}[Proof of Theorem $\ref{ab9}$] If $G$ is of type $A_n$ then the result is contained in Theorem \ref{pt7}. For $G$ of type $E_6$ and $E_7$ see Lemmas \ref{e65} and   \ref{77}, respectively. The case with $G=G_2$ is covered by   Corollary \ref{gk6}. 

Suppose first that $V=V_1\otimes V_2$ is tensor-decomposable, that is, both $V_1,V_2$ are non-trivial. 
Then $s$ is cyclic on $V_1$ and on $V_2$  (Lemma \ref{ma2}(4)). 
 If $G$ is not of type $A_n,n>1$, $D_n, n$ odd, or $E_6$ then every semisimple element of $G$ is real and the result follows from Lemma \ref{ma2}(4). 
 
Let $G  =D_n, n>3$ odd. Let $\om=\sum p^{i}\mu_i$ be the highest weight of $V$, where $\mu_i$'s are $p$-restricted.
Then $V=\otimes_{i}V_{p^i\mu_i}$. By Lemma \ref{gc3}, the weights of $V_{p^i\mu_i}$, and hence of $V_{\mu_i}$ are of \mult 1, so, by Table 2,
 $\mu_i\in \{\om_1,\om_{n-1},\om_n\}$.   
Then $m_s(V)\leq 2$ by Lemma \ref{tp7}.

Next suppose that $V$ is tensor-indecomposable. By Steinberg's theorem,  $V$ is a Frobenius twist of a $p$-restricted $G$-module.  So we can assume that $V$ itself is $p$-restricted. 

Suppose first that $\om$ occurs in Table 1. Then $\dim V^T>1$ and $\dim V^T=\dim V^s=m_s(V)$ by Theorem \ref{do1}(1), so (2) follows from the data in Table 1. 

Suppose that $\om$ does not  occur in Table 1. By  Theorem \ref{ag8}, $\om$ occurs in Table 2.  
 If $(G,p,\om)=(F_4,3,\om_4)$ then the result  follows from  Lemma \ref{ff44}.
This completes considerations of groups of  the exceptional Lie types. 

So we are left with classical groups (other than of type $A_n$). 
The case with $G=D_n, n\geq 4$ is settled in Lemma \ref{oo1} for $\om=\om_{n-1}$ or $\om_n$, for  $\om=\om_1$ see Lemma  \ref{co1}(3). 
This implies the result for an arbitrary $p$-restricted $\om$ in view of Theorem \ref{ag8}  and Table 2. A similar argument works for $G=B_n$,  $n>2,p\neq  2$, with use of Lemmas \ref{cn5} and \ref{o11}. The case $B_2\cong C_2$ is considered in Lemma \ref{2bc}. 

The case  $G=C_n, p=2$ has been examined in Lemma \ref{o11}. 
Suppose that $G$ is of type $C_n$ with $p\neq 2$.   By Theorem \ref{ag8}  and Tables 1,2 we have to inspect the cases with $\om\in \{\om_1,\om_{n-1}+\frac{p-3}{2}\om_{n},\frac{p-1}{2}\om_{n}\}$
 as well as $n=3, p\neq 2, \om=\om_3$. The case with $\om=\om_1$ is trivial, the other cases are dealt with in Lemmas \ref{cn5}, \ref{5n5} and   
the remaining case (where $G=C_3,p\neq 2, \om=\om_3)$ is settled in Lemma \ref{cn3}.
\end{proof}  

 One easily observes that the exceptions in Theorem \ref{ab9} are genuine. It suffices to show this for $\om\in\{3\om_1,p\neq 2, (1+p^k)\om_1\}$.  
 Let $\diag(a,a\up)\in SL_2(F)$ be the matrix of $s$ on $V_{\om_1}$. Then the matrices of $s$ on $V_{2\om_1}$, $V_{3\om_1}$ and $V_{(1+p^k)\om_1}$ are
 $\diag(a^2,1,a^{-2})$,  $\diag(a^3,a,a\up, a^{-3})$  and   $\diag(a^{p+1},a^{p-1},a^{-p-1},a^{-p+1})$, respectively. If  $a^2=-\Id$,  $a^3=1\neq a$   then $s$
is almost cyclic on $V_{2\om_1}$, $V_{3\om_1}$, respectively, and $m_s(V_{2\om_1})=m_s(V_{3\om_1})=2$. If $a$ is a primitive $(p+1)$-root of unity then 
 $s $ is almost cyclic on  $V_{(p+1)\om_1}$ and $m_s(V_{(p+1)\om_1})=2$.

\bp[Proof of Theorem {\rm \ref{ib2}}] The result follows from Theorem \ref{ab9}. \enp

\begin{proof}[Proof of Corollary {\rm \ref{hv2}}] Suppose the contrary, that $m_s(V)\geq \dim V/2>1$. 
By Theorem \ref{ab9}, either $m_s(V)\leq 2$, hence $\dim V\leq 4$,
or $V$ is as in item $(1),(2)$ or $ (3)$ of Theorem \ref{ab9}. 

Suppose first that $\dim V\leq 4$. 
If $\dim  V=3$ then $G=A_1$, 
$p\neq 2$, $\om=2p^i\om_1$ or  $G=A_2$, 
 $\om\in \{p^i\om_1, p^i\om_2\}$. In the former case  $m_s(V)=2$ for $s\in SL_2(F)$ implies $s^2=-\Id$ as recorded in the corollary. The latter case is ruled out by Lemma \ref{co1} or by  a straightforward calculation. Let $\dim V=4$.  
Then   $G\in\{A_1, A_3, B_2,C_2\}$. In fact, if $G=A_3$ then $V$ is a twist of the natural $G$-module, on which $m_s(V)\leq 1$ for a regular semisimple element  $s\in G$ (Lemma \ref{co1}). 
The cases with $G=B_2$ and $C_2$ are ruled out by  Remark \ref{bc4}. Let $G=A_1$. Then $\om$  is well known to be either $3p^i\om_1\,\,(p\neq 2,3),
$ or $(p^i+p^j)\om_1$ as recorded in the statement. 

The result is obvious if item (3) of Theorem \ref{ab9}  holds as $\dim V\geq 27$ for $G=E_6$.

Suppose that (2) of Theorem \ref{ab9} holds. Then $m_s(V)=\dim V^T$ and inspection of $\dim V/\dim V^T$ in Table 1 yields a contradiction.
 
\st item (1) of Theorem \ref{ab9} holds so $G$ is of type $A_n$. We can assume $n>1$ as otherwise $m_s(V)\leq 2$ by Theorem \ref{ib2} and $\dim V\leq 4$; this case has been examined above.  In addition,  we can assume that $V$ is not a twist of the adjoint $G$-module as this is included in case (2), unless $n=2,p=3$. In the latter case $\dim V=7$ and $m_s(V)\leq 2$ by Theorem \ref{ib2}.

By the above, we can assume that $m_s(V)>2$ and $\dim V> 5$.

So we assume that $n>1$. If $V$ is a twist of the \ir $G$-module with \hw $\om_1$
 then we get a contradiction with Lemma \ref{co1}. 

If   $V$ is {\it  not} a twist of the \ir $G$-module with \hw $\om_1$ then $\dim V\geq n(n+1)/2$
so $\dim V/m_s(V)\geq n/2$ and $n/2>2$ for $n>4$.  Suppose that $n\leq 4$ and $\dim V\leq 2(n+1)$. If $n=2$, this implies $p\neq 2$ and  $\om\in\{2\om_1,2\om_2\}$. However, 
in these cases $\dim V=6$ and $m_s(V)\leq 2$ by Lemma \ref{we4}, so this case is ruled out.    
 By \cite[Appendices A6,A7,A8]{Lu},  the highest weight $\om$ of $V$ is 
  $ p^t\om_2$   if $n=3$ and $\{ p^t\om_2$ or $p^t\om_3\}$ if $n=4$ for some integer $t\geq 0$. So either  $G=SL_4(F)$, $V=V_{\om_2}$ or $G=SL_5(F)$, $V=V_{\om_2}$ or $V_{\om_3}$. As $V_{\om_2}$, $V_{\om_3}$ are dual for $n=4$,
it suffices to examine $\om=\om_2$ in both the cases. If  $n=3$ then $V$ is a twist of the natural module for   $D_3\cong A_3$ on which $m_V(s)\leq2$ as $s$
is regular (see Lemma \ref{co1}). If $n=4$ then $ m_s(V_{\om_2})\leq 2$ by Lemma \ref{ep9} and $\dim V_{\om_2}=10$, so this case is ruled out. This completes  the proof. \enp


\section{An application to a recognition problem for linear groups over finite fields}

There is a problem in the theory of linear groups over finite fields aimed to determining the subgroups $H$ of $GL_n(q)$ containing a matrix $\diag(\Id_m,M)$, where $M\in GL_{n-m}(q)$ is \irt The latest and most significant contribution to this project is made in \cite{GPPS}, where the authors solve this problem for $m<n/2$. The result of \cite{GPPS} can be also interpreted as a classification of subgroups of $GL_n(q)$ containing a cyclic Sylow $\ell$-subgroup of $GL_n(q)$ for some prime divisor $\ell$ of $GL_n(q)$, see \cite{BP}.
Observe that  the case where $H$ is \ir in  $GL_n(q)$ is the bulk of the problem. 

In this section we contribute to the problem in the case  where $H$ is a quasi-simple group of Lie type in defining characteristic $p>0$ for $p|q$. 
Our purpose is to consider the complementary case with $m\geq  n/2$. Observe that the case with $p\not|q $ is considered in \cite{DZ3}.

Note that specializing the   main theorem of \cite{GPPS} to the case where   $H$ is a quasi-simple  group of  Lie type 
leads one to Examples 2.1, 2.4 and 2.9 in \cite{GPPS}. Then it is easy to observe that Examples 2.1, 2.4  yield  nearly natural \reps of $H$ (whereas Examples 2.9 are irregular cases arising for $n\leq 9$). 
The converse is also true, that is, if $H$ is a quasi-simple  group of  Lie type  in defining characteristic $p|q$ and
$\rho: H\ra GL_n(q)$ is  a nearly natural \irr then $\rho(H)$ is as one of these examples. However, we do not provide here the proof of this fact.

\subsection{Representation of $SL_2(q)$ containing almost cyclic elements}
Our primary goal here is to give an example with $m= n/2$ (see Lemma \ref{ev3}), which is essential for what follows. 

\bl{s1t} Let $\mathbf{G}=SL_2(F)$, let   $\rho$ be an \irr of  $\mathbf{G}$ with \hw  $(1+p^{i})\om_1$ for some integer  $i\geq 1$. Let $s\in \mathbf{G}$ be a non-central semisimple element such that $\rho(s)$ is almost cyclic   and $m_s(\rho)>1$. 
Then  $\rho(s)\in PSL_2(p^i)\subset GL_4(p^i)$ and s is conjugate in $\mathbf{G}$  to an element of $SL_2(p^i)$, unless  $p\neq 2$ and $s^2$ is conjugate to an element of $SL_2(p^i)$. In addition, the group $\rho(SL_2(p^{2i}))$ is \ir  whereas $\rho(SL_2(p^i))$ is reducible.\el

\bp  Let $s=\diag(a,a\up)\in SL_2(F)$. Then $\rho(s)=\diag(a^{p^ i+1},a^{p^ i-1}, a^{-p^i+1},a^{-p^ i-1})$, and some of these \eis coincide. 
Here $a\in F^\times, $ $a^2\neq 1$ as  $s\notin Z(\mathbf{G})$. This implies $a^{p}\neq a^{-p}$, $a^{p^ i+1}\neq a^{p^ i-1}$ and $a^{-p^ i+1}\neq a^{-p^ i-1}$.  
 So either $a^{2(p^ i+1)}=1$ or $a^{2(p^ i-1)}=1$, and hence either $s^{p^ i+1}\in Z(\mathbf{G})$ or $s^{p^ i-1}\in Z(\mathbf{G})$. If $p=2$ then $Z(\mathbf{G})=1$, and the group $SL_2(2^i)$ contains elements of  order $|s|$;  these are obviously conjugate to elements of $\lan s\ran$.  
Let $p>2$.  If $a^{p^ i+1}=1$ or $a^{p^ i-1}=1$ then, similarly, $s$ is conjugate to an element of $SL_2(p^{i})$.
Let   $a^{p^ i+1}=-1$ or $a^{p^ i-1}=-1$; then    $\rho(s)=(-1,-1,-a^{-2},-a^{2})\in GL_4(p^i)$ as $s^2\in SL_2(p^i)$.  
 So $\rho(s)\in GL_4(p^i)$, as claimed.  

Turn to the additional claim of the lemma.  If $q=p^{2i}$ then the weight $(1+p^i)\om_1$ is $q$-restricted, so $\rho$ remains \ir under restriction to $SL_2(q)$  (see \cite[\S 2.11]{Hub}). If $q=p^{i}$ then $\rho|_{SL_2(q)}\cong (\tau\otimes\tau)|_{SL_2(q)}$, where $\tau$ is the \irr of $\mathbf{G}$ with \hw  $\om_1$. It is well known that $\tau\otimes\tau$ is reducible. 
\enp

\bl{s2s} Let $\mathbf{G}=SL_2(F)$, $G=SL_2(q)$, $q=p^d>p$, and let   $\rho$ be an \irr of  $\mathbf{G}$ with \hw  $(1+p^{i})\om_1$ for some integer  $i\geq 1$. Suppose that $\rho(G)$ is conjugate in $GL_4(F)$ to an \ir subgroup of $GL_4(p^m)$, where $m$ is minimal possible. Then $d\not|i$ and either d is even, $i\pmod{d}=d/2$ and $m=d/2$ or  $i\pmod{d}\neq d/2$ and $m=d$.
\el 

\bp  Obviously,  $\rho(G)\subset GL_4(q)$.  
Suppose that the group $\rho(G)$ is conjugate to a subgroup of $GL_4(F_0)$, where $F_0$ is a  subfield of $\FF_{q}$ and  $|F_0|=p^m$. Then  $m|d$.
(As $ \FF_{q}$ is a vector space over $F_0$, we have $ |\FF_{q}|=|F_0|^j$ for some integer $j>0$, 
so $q=p^{d}=(p^m)^j$, whence $mj=d$.) As $\rho(G)$ is irreducible, $i$ is not a multiple of $d$. (Otherwise $\rho|_H\cong V_{\om_1}|_H\otimes V_{\om_1}|_H$ is reducible.)

\st $F_0\neq  \FF_{q}$, so $m<d$. 

By Zsigmondy's theorem (see \cite[5.2.14]{KL}), if $q^{2}=p^{2d}>p^2$ (this holds by assumption) and $q^2\neq 64$ then there is a prime $\ell|(q^{2}-1)=p^{2d}-1$ such that    $(\ell, p^j-1)=1$ for every $j<2d$. Therefore, if $F_0\neq \FF_{q}$ (hence $m<d$) and $q\neq 8$   then 
$(\ell, p^{2m}-1)=1$. Moreover,   $(\ell, p^{4m}-1)=1$.
Indeed, otherwise $\ell$ divides $(p^{4m}-1,p^{2d}-1)=p^{(4m,2d)}-1$ by \cite[Hilfsatz 2(a)]{Hup}. As $(4m,2d)\leq 2d$,
the claim is true unless $(4m,2d)= 2d$, whence $d|2m$, and hence    $2m=d$. Therefore, $\FF_{q}$ is the quadratic extension of $F_0$.
Let $q^2 = 64$ so  $G=SL_2(8)$, $d=3$ and $|s|=9$.  Then $\rho(G)\subset GL_4(8)$. As $2^m|8$,  it follows that 
either  $2^m=8$, as required,  or $m=1$; in the latter case $\phi(G)\subset GL_4(2)$. However,  $GL_4(2)$ has no element of order 9, so we have a contradiction. 

    Finally, we show that $i$ is a multiple of $d/2$. Let $\gamma: F\ra F$ be the mapping defined by $x\ra x^{p^m}$ for $x\in F$. This extends to the mapping $\tau:GL_4(F)\ra  GL_4(F)$. Then $\tau(y)=y$ for $y\in GL_4(F_0)$ and $\tau$ acts as an \au of order 2 on $GL_4(q)$. Then the \hw of the \rep $\rho^\tau:h\ra  \tau(\rho(h))$ for $h\in SL_2(F)$ is $p^m(1+p^i)\om_1$. As $\tau$ acts trivially on $GL_4(F_0)$ and hence on $\rho(G)$, it follows that $i$ is a multiple of $d/2$, as claimed.

Let $F_0= \FF_{q}$.  By the above, $d$ does not divide $i$, and we are left to show that 
$i$ is not a multile of $d/2$. Suppose the contrary.    Then   the mapping $x\ra x^{p^i}$ for $x\in F$ yields
the Galois \au of $\FF_q$ of order 2. \itf the trace of every element $\rho(h), h\in H$ lies  in
the subfield $F_1$, say, of index 2 in $\FF_q$. Then $\rho(H)$ can be realized over $F_1$,
a contradiction.
\enp

\begin{corol}\label{os4} Let $K=GL_4(p^m)=GL(V)$, where  $m>1$. Then for every $j|2m$, $j>1$, 
 there is an absolutely \ir subgroup $G_j\cong  PSL_2(p^j)$ of K.
In addition we have:

$(1)$ if $p\neq 3$ and $t\in G_j$, $j>1$, is of order $(p+1)/(2,p+1)$ then t is almost cyclic and $m_t(V)= 2;$ 

$(2)$ if $p= 3$, $j$ is even and  $t\in   G_{j}$  of order $4$   then t is almost cyclic and $m_t(V)= 2$. 
\end{corol}

\bp Let $\mathbf{G}$ be as in Lemma \ref{s1t}, and  let $\rho $ be an \irr of $\mathbf{G}$ with \hw $(1+p)\om_1$. If $j|m$ then $G_j:=\rho(SL_2(p^j))\cong  PSL_2(p^j)$ is
 \ir for $j>1$ by Lemma \ref{s1t}. Obviously, $G_j\subset GL_4(p^j)\subset GL_4(p^m). $ If $j\not | \, m$ then $j$ is even. Take  $i=j/2$ in  Lemma \ref{s1t}. Then $G_j=\rho(SL_2(p^j))$ is \ir
and  $G_j\subset GL_4(p^{j/2})\subset GL_4(p^m) $ as $j/2$ divides $m$.

If $p\neq 3$ then let $s\in SL_2(p)\subset SL_2(p^j)$ be of order $p+1$. By Lemma \ref{s1t}, $s$ is almost cyclic on $V$ and $m_t(V)= 2.$ 

If $p=3$ then  $SL_2(9)\subset SL_2(p^j)$ has an element of order 8. As above,  take  $i=1$ in  Lemma \ref{s1t}. Then $t=\rho(s)\in G_j$ is   is almost cyclic and $m_t(V)= 2$. 
\enp

\begin{rem}\label{x6x}{\rm  It is not true that the \ir subgroups of $GL_4(F)$ isomorphic to $G_j\cong PSL_2(p^j)$
are conjugate even if they contain elements $s$ with $m_s(V)=2$. Indeed, let $\rho_i$ be an \irr of $\mathbf{G}\cong SL_2(F)$ with \hw $(1+p^i)\om_1$, $i|j$, $i< j$. Set $G_{ij}=\rho_i(SL_2(p^j))$. Then  $G_{ij}$ is \ir and $G_{ij}\cong G_j$. 
Let $s_i\in SL_2(p^j)$ be of order $p^i+1$ and $t_i=\rho_i(s_i)\in G_{ij}$. Then $t_i$ has \ei 1. Let $i'|j, $ 
and $i\neq i'<j$. Then  $\rho_{i'}(s_i)\in G_{ij}$ does not have \ei $1$. Therefore, $\rho_i $ and $\rho_{i'}$ are not equivalent  and, moreover, $  G_{ij}$ and $G_{i'j}$ are not conjugate in $GL_4(F)$.  }
\end{rem}

\bl{es3} Let q be a p-power, and $r\geq 1$ an integer. If $r\neq 1$ then there exists an odd prime $\ell$ such that $\ell$ divides $q^r+1$ and does not divide $q^j-1$ for $j\leq r$. This remains true for $r=1$ unless $q=8$ or $q+1$ is a $2$-power. In the exceptional cases $q+1$ is an $\ell$-power. \el
 
\bp By  Zsigmondy's theorem (see \cite[5.2.14]{KL}), if $q^{2}=p^{2d}>p^2$  
and $q^2\neq 64$ then there is a prime $\ell|(q^{2}-1)=p^{2d}-1$ such that    $\ell\not|(p^j-1)$ for every $j<2d$. Clearly, $\ell$ is odd here.
If $q^2= 64$ then take $\ell^2=9$.   If $q=p$ then  let $\ell$ be the maximal prime dividing $q+1$. If $\ell\neq 2$ then $\ell\not|(q-1)$ as desired. Otherwise, $\ell=2$ and $q+1$ is a 2-power, and the claim follows.\enp

\bl{ev3} Let $H=GL_{4r}(q)=GL(V)$, $q^r\neq 3$. 
Then there exists an \ir  subgroup $Y\cong PSL_2(q^{2r})$  of H  and $h\in PSL_2(q^{r})\subset Y$ such that $|h|$ is a prime power, h is almost cyclic on V and 
the \mult of \ei $1$ of h equals $2r$. Moreover, $V=V_1+V_2$, where $V_1$ is the $1$-eigenspace of $h$, $hV_2=V_2 $ and $h$ acts irreducibly on $V_2$.
If $r\neq 1$ then h can be chosen of  prime order. \el

\bp In assumptions of Lemma \ref{s1t} set $G=SL_2(q^{2r})$ and $i=r$. By Lemma \ref{es3}, there exists 
 $s\in SL_2(q^{r})\subset G$ be such that $|s|$ is a prime power,  $|s|$  divides $q^{r}+1$
and does not divide  $ q^{2t}-1$ for any positive integer $t<r$. If $r\neq 1$ then we can choose $s$ to be of a prime order.  In addition, we can choose $s$ to be of odd order unless $r=1$ and $q+1$ is a 2-power.

Set $h=\rho(s)$.  
By Lemma \ref{s1t}, $\rho(G)$ is \ir and  contained in $GL_4(q^{r})$ (see Lemma \ref{s2s}). 
Let $V'$ be the underlying space of $ GL_4(q^{r})$. Then $h=\rho(s)$ is conjugate in $GL_4(q^{r})$ to a matrix $\diag(1,A,1) $, where $A\in GL_2(q^r)$. As $\det A=1$,  it follows  that
  $A$ does not have \ei 1.  Therefore,    
 $V'=V'_1+V'_2$, where $V'_1$ is the 1-eigenspace of $h$ on $V'$ and $sV'_2=V'_2 $.

Next embed $GL_4(q^{r})$ into $GL_{4r}(q)$ by viewing  $\FF_{q^{r}}$ as a vector space over  $\FF_{q}$. 
Then  $V'_1,V_2'$  and $V'$ can be viewed as
vector spaces  $V_1, V_2, V$, say,   over $\FF_q$, and we have  $\dim V_1=\dim V_2=2r$. We claim that $h$   is \ir on $V_2$. Indeed,   
if $h$ is reducible on  $V_2$ then  $|h|$ divides $q^j-1$ for $j<2r$ and also $q^{2r}-1$, hence $q^{(2r,j)}-1$ by \cite[Hilfsatz 2(a)]{Hup}. So we can assume $j|2r$, hence $j\leq r$. 
As the order of $h$ is equal to $|s|$ or $|s|/2$, we conclude that   $h$ is \ir on $V_2$ unless $|h|=|s|/2$, in particular, $|s|$ is a 2-power. This happens only when $q+1$ is a 2-power
and then we can choose $s$ with  $|s|=q+1$. Then $(q-1)/2$ is odd, so  $|h|=(q+1)/2$ divides $q-1$, which  implies $q+1=4$. This contradicts the assumption. 
\enp

\subsection{Regular semisimple elements in representations of quasi-simple finite groups of Lie type}

Recall that the notions of  cyclic and almost cyclic elements are meaningful for elements $g\in GL_n(F_1)$ for any subfield $F_1$ of $F$, that is, we simply   apply the definition to $g$ as an element of $GL_n(F)$. 

\bl{pc1} Let $G$ be a finite quasi-simple group of Lie type in defining characteristic $p>0$,  let $V$ be a non-trivial  \ir $\FF_qG$-module  for some p-power $q$, $\dim V>2$ and $\overline{V}=V\otimes F$. Let $s\in G$ be a   regular semisimple element and let $m_s(V)$ be the maximal  \ei multiplicities of $s$ on $\overline{V}$.
Suppose that $s$ is almost cyclic on $\overline{V}$ and $m_s(\overline{V})\geq \dim V/2.$ Then   $G$ is isomorphic to $ SL_2(q')$, or $PSL_2(q')$ for some p-power  $q'$. Let $\rho$ be the \rep afforded by V. Then we can assume that $G\cong SL_2(q')$   and
one of the \f holds:

 $(1)$ $\dim V=3$, $\rho(G)\cong  PSL_2(q')$,  $q$ is a power of $q'$, $p\neq 2$ and $s^2\in Z(G)$;

$(2)$ $\dim V=4$,  $p\neq 2,3$, $\rho(G)\cong  SL_2(q')$, $q$ is a power of $q'$ and $s^3\in Z(G);$  

$(3)$ $\dim V=4$,  $\rho(G)\cong PSL_2(q')$, $q'>p$,  $q^2$ is a $q'$-power, and 
 if $q'=3^l$ then  $l$ is not  a prime; 

$(4)$ $\dim V=4r$, $r>1$, $\rho(G)\cong  PSL_2(q')$, $ q^{2r}$  is a power of $q'$,   
and $m_s(V)= \dim V/2.$
\el

\bp Let $\mathbf{G}$ be a  simply connected simple \ag  such that $G=\mathbf{G}^{\si}$ and  $\si$ is some Steinberg endomorphism of $\mathbf{G}$.  Saying that $s\in G$ is a   regular semisimple element means that $s$ is regular in $\mathbf{G}$. 

(i) Suppose first that  $V$ is absolutely \irt Then  $V $ extends to a $\mathbf{G}$-module $\overline{V}$.  
By Corollary  \ref{hv2}, $\mathbf{G}$ is of type $A_1$, so $\mathbf{G}\cong SL_2(F)$  and $3\leq \dim \overline{V}=\dim V\leq 4$. 
We will also use  $\rho$ to denote the corresponding \rep of $\mathbf{G}$.

Then  $G\cong SL_2(q')$  for some $p$-power $q'$.
  Let $s\in G$ and let $\diag(a,a\up)$ with  $a^2\neq 1$ be a  conjugate of $s$ in $ SL_2(F)$.

  If $\dim \overline{V}=3$ then $p\neq 2$ and  $\overline{V}$ is a Frobenius twist of the \ir $\mathbf{G}$-module with \hw $2\om_1$; the matrix of $\rho(s)$  is $\diag(a^2,1,a^{-2})$, which is cyclic unless $a^2=a^{-2}$, and hence $s^2=-\Id\in Z(G)$. One observes that such $s$ is conjugate to an element of $SL_2(p)$. So $s\in SL_2(q')$ whenever $SL_2(q')\subseteq SL_2(q)$, equivalently,   $q$ is a $q'$-power.

Let $\dim \overline{V}=4$. Then   $\overline{V}$  is a Frobenius twist of the \ir $\mathbf{G}$-module with \hw  $3\om_1$ (for $p\neq 3$) or $(1+p^i\om_1)$ for some $i>0$. By Lemma \ref{ft1}, we can assume that 
the \hw of $\overline{V}$ is $3\om_1$ or $(1+p^i\om_1)$.
In the former case  $\rho(G)\cong SL_2(q')$ and  the matrix of $s$ on $\overline{V}$ is $\diag(a^3,a,a\up, a^{-3})$. If this matrix is not cyclic then either $a^6=1,p\neq 3$, and hence $s^3\in Z(G)$, as required,  or $a^4=1$ in which case $s$ is not almost cyclic. In the latter case $q'>p$ , $\rho(G)\cong PSL_2(q')$ and the matrix of $s$ on $\overline{V}$
is $\diag(a^{p^i+1},a^{p^i-1},a^{-p^i+1}, a^{-p^i-1})$.   If this matrix is not cyclic then either $a^{2p^i+2}=1$ or $a^{2p^i-2}=1$ or $a^{2p^i}=1$; the latter implies $a^2=1$, which is false as  $s$ is regular. If $a^{2p^i+2}=a^{2p^i-2}$ then $a^4=1$, whence $p\neq 2$ and $a^2=-1$. One checks that $s$ is not almost cyclic on $\overline{V} $ in this case. In two other cases $s$ is almost cyclic on $\overline{V} $. By  
Lemma \ref{s2s}, $\rho(G)$ is contained in $ GL_4(q')$ or $GL_4(\sqrt{q'})$, whence $q^2$ is a power of $q'$. If $p\neq 3$ then  
$SL_2(q')$ has an element $s$  in question for every $q'$ (see Corollary \ref{os4}), so (3) holds in this case (provided $q'>p$).

 Let $p=3$ and let $q=3^m$ for some $m>1$.  By  Corollary \ref{os4},  if $q'$ is a square then 
$SL_2(q')$ contains an element $s$ with properties required.  \st $q'$ is not a square. Then  $q'=3^l$ with $l$  odd, and  $l>1$  (as $q'>p=3$). By Lemma \ref{s2s},  $q'$ is a minimal 3-power $x $ such that  $\rho(SL_2(q'))\subseteq GL_4(x)$. So $GL_4(q')\subseteq GL_4(q)$,  and hence $q$ is a power of $q'$.

\st $l$ is a prime.   By Lemma \ref{s1t}, $s$ is contained in a subgroup isomorphic to $SL_2(p^i)$, so $|s|$ divides $3^{2i}-1$, and also $3^{2l}-1$ as $s\in SL_2(3^l)$. By
 \cite[Hilfsatz 2(a)]{Hup}, $|s|$ divides $3^{(2i,2l)}-1$. As $l$ is a prime, $(2i,2l)\in\{2,l,2l\}$, so either $(i,l)=1$ or $l|i$.   
In the latter case $\rho(SL_2(q'))$ is reducible by Lemma \ref{s2s}. So  $(i,l)=1$ so $|s|$ divides  $3^2- 1=8$. As $l$ is odd, we have $(3^l\pm 1,8)\neq 4$,whence  $|s|=4$,  but then $\rho(s)$ is not almost cyclic.

Suppose that $l$ is not a prime, and let $l'$ be a prime with $l'|l$. 
Let $i=l'$ and let $s\in SL_2(q')$ be of order  $3^{l'}+1$. Then  the group $\rho(SL_2(q'))$ is  irreducible by Steinberg's theorem (see \cite[\S 2.11]{Hub}), $\rho(s)$ is \ac and $m_s(\overline{V})=2 $  by Lemma \ref{s1t}. This completes the proof of the case with $\dim \overline{V}=4$.

(ii) Suppose that $V$ is not absolutely irreducible.
Then, by Schur's lemma,  the centralizer of $\rho(G)$ in ${\rm End}(V)$ is a field $F_0$, say, 
and $|F_0|=q^r$ for some $r>1$. \itf $\rho(G)\subset GL_{n/r}(q^r)\subset GL(V)$ (\cite[Lemma 2.10.1]{KL}). 

Therefore,   $\overline{V}=V\otimes F$ as an $FGL_{n/r}(q^r)$-module is a direct sum $\overline{V}_1\oplus \cdots \oplus \overline{V}_r$, where $\overline{V}_1\ld \overline{V}_r$ are \ir $FGL_{n/r}(q^r)$-modules that are Galois conjugate to each other (by ${\rm Gal}(\FF_{q^r} /\FF_q)$), and each $\overline{V}_1\ld \overline{V}_r$ is absolutely  irreducible as $F G$-module,  see \cite[Lemma 2.10.2]{KL}. Then $\dim V=r\dim \overline{V}_1$.
Note that 
$\dim \overline{V}_1>2$. (Indeed, if $\dim \overline{V}_i=2$ then the \eis of $s$ on $\overline{V}_i$ are $a_i,a_i\up$, where  $\pm 1\neq a_i\in F$.
As $m_s(\overline{V})>1$, we have $a_i\in\{a_j,a_j\up\}$, hence $s $ is not almost cyclic on $V$, a contradiction.) So $r<\dim V/2$.


Clearly, $s$ is almost cyclic on each $\overline{V}_1\ld \overline{V}_r$. 
As these are Galois conjugate to each other, it follows that $m_s(V)\leq m_s(\overline{V}_1)+\cdots +m_s(\overline{V}_r)=r\cdot m_s(\overline{V}_1)$.
 
As above, each 
$\overline{V}_1\ld \overline{V}_r$ extends to a $\mathbf{G}$-module.
  By Corollary  \ref{hv2},  $m_s(\overline{V}_i)< \dim \overline{V}_i/2$ for every $i=1\ld r$ (which is a contradiction so the result follows), unless possibly when $\mathbf{G}$ is  of type $A_1$  and $\dim \overline{V}_1=3,4$. So we can assume that $\mathbf{G}\cong SL_2(F)$ and then $G=SL_2(q')$ for some $p$-power $q'$. 

Since $s$ is almost cyclic on  $\overline{V}$ and $s,s\up$ are conjugate in $G$, the \ei of $s$ of \mult greater than one is 1 or $-1$ (Lemma \ref{ma2}(4)).

 If $\dim \overline{V}_1=3$ and $m_s(\overline{V}_1)>1$ then $s^2\in Z(\mathbf{G})$, and $s$ is almost cyclic on $V$ implies $\dim V=3$, which is a contradiction as $r>1$. 

Let $\dim \overline{V}_1=4$. If the \hw of $\overline{V}_1$ is $3p^i\om_1$ with $p\neq 2,3$
then $s^3\in Z(\mathbf{G})$.  
Then $s$ is conjugate to an element of a subgroup $SL_2(p)$. The restrictions of $\overline{V}_j$ to $SL_2(p)$ for $j=1\ld r$ yield isomorphic modules. As $s$ is almost cyclic and $r>1$, this is a contradiction.

Suppose that the \hw of $\overline{V}_1$ is $(1+p^i)\om_1$. Then $\rho_j(\mathbf{G})\cong PSL_2(F)$ for every $j=1\ld r$, so $Z(G)$ acts trivially on $V$.

By the above, $\overline{V}_1$ is an \ir $FGL_{n/r}(q^r)$-modules and absolutely  irreducible as $F G$-module,
that is, $\rho_1(G)=\rho_1(SL_2(q'))$ is absolutely \ir subgroup of $GL_4(q^r)$. By (3), $q^{2r}$ is a power of 
$q'$, as claimed in (4). In addition, $m_s(\overline{V}_1)\in\{1,2\}$, so   $m_s(\overline{V})\in\{1,2r\}$ as required.    \enp

Observe that   the case (4) is genuine. Indeed, 
by Lemma \ref{ev3}, if $q^{2r}=q'$ then $GL_{4r}(q)$ contains an \ir subgroup isomorphic to $PSL_2(q')$ in which some element  $s $  of prime order is \ac on $V$, $m_s(V)=\dim V/2$ and $s$ is \ir on $(1-s)V$.

\begin{rem} {\rm One can wish to further refine item (4) of Lemma \ref{pc1} by determining $q'$ and the \ir groups $X\cong PSL_2(q')\subset GL_{4r}(q)$ up to conjugacy. However, we shall not work in this direction.
 Observe that, in notation of Lemma \ref{pc1}(4), 
$\rho(G)$ remains \ir in some larger groups $GL_{4r}(q^t)$. 
For instance, let $t$ be an arbitrary  prime such that $(4r,t)=1$. If  $\rho$ is reducible in $GL_{4r}(q^t)$ then $\rho(PSL_2(q'))=\rho_1(PSL_2(q'))\oplus \cdots \oplus \rho_t(PSL_2(q'))$, where  $t|4r$. This is a contradiction.
}\end{rem}

\begin{proof}[Proof of Theorem {\rm \ref{hv3}}] The main assertion of the theorem follows from Lemma \ref{pc1}. 

We will then prove  the converse claim. If $4|\dim V$ then this follows from Lemma \ref{ev3}. If $\dim V=3$, $p\neq2$ then consider an \irr $\rho$ of 
$SL_2(F)$ with \hw $2\om_1$, and set  $\rho'=\rho|_{SL_2(q')}$. Then for $s\in SL_2(q')$ with $s^2=-\Id$ we have $\rho'(s)=\diag(-1,1,-1)$ under a suitable basis of $V$.\enp

\begin{corol}\label{co5} In assumptions of Lemma {\rm \ref{pc1}} suppose that $q=2$. Then item $(3)$ or $(4)$ 
Lemma {\rm \ref{pc1}} holds  and $q'=2^{2r}$.  In addition, $\rho(G)\subset SO^-_{4r}(q)$.\end{corol}
 
\bp It is obvious that (1) and (2)  of Lemma {\rm \ref{pc1}} do not hold. Let $q'$ be as in Lemma \ref{pc1}. If (3) holds then $q'|4$ and $q>2$,
so $q'=4$ as required. Let (4) hold. Suppose the contrary, that $q'<2^{2r}$. 
  We have seen in the proof of Lemma \ref{pc1} that $V\otimes F=\overline{V}_1\oplus \cdots \oplus \overline{V}_r$,
where       $\overline{V}_1\ld \overline{V}_r$  are \ir $SL_2(F)$-modules of dimension 4, $s$ is \ac on each $\overline{V}_l $  and $m_s(\overline{V}_l)=2$ for $l=1\ld r$. 
Let $\om= (2^j+2^i)\om_1$ be the \hw of $V_1$, where  $1\leq  j<i$. By Lemma \ref{ft1}, 
we can assume here that  $j=1$. 
By Lemma \ref{s1t}, $s$ is conjugate to an element of $SL_2(2^i)$ and $\rho_l(s)$ is conjugate to an element of $GL_4(2^i)$, where $\rho_l$ is the \rep afforded by $\overline{V}_l$. As $s$ is \ac on $V$,
the modules $\overline{V}_l|_{SL_2(2^i)}$ for $l=1\ld r$ are pairwise non-isomorphic. In addition, these are twisted (or conjugate) to each other by elements of the Galois group ${\rm Gal}(\FF_{2^r}/\FF_2)$. \itf $i=r$ and $\om=(1+2^r)\om_1$. So $q'$ is a power of $2^r$ and $2^{2r}$ is a power of $q'$ by Lemma \ref{pc1}(4).  Here $q'\neq 2^r$ as $\rho_l(G)$ is \ir (see Lemma \ref{s1t}).
So $q'=2^{2r}$ and $G=SL_2(2^{2r}) $ as claimed. In addition, $\rho_1(G)\subset O^-_4(2^r)$, and hence 
$\rho(G)\subset O^-_{4r}(2)$ by \cite[Hilfsatz 1(d) and  Bemerkung at page 148]{Hup}.   \enp

\subsection{Non-regular semisimple elements in representations of quasi-simple finite groups of Lie type}

In order to deal uniformly with $p=2$ and $p>2$ we denote by $SO_8(F) $ the connected component of the orthogonal group $O_8(F)$.

\bl{d43} Let $s\in {\mathbf{G}}=Spin_8(F) $ be a non-central semisimple element, 
and V the natural $G$-module. Suppose that s is almost cyclic on V. 

$(1)$  If $\lam$ is an \ei of s on V of \mult $m_s(V)>1$ then $\lam\in\{1,-1\}$ and either $1$ or $-1$ is not an \ei of $s$.

$(2)$ Let $G\cong {}^3D_4(q)$ and $s\in G$. 
Then s is regular.\el

\bp  Set $\mathbf{H}=SO_8(F)$ and let $h:=\mathbf{G}\ra \mathbf{H}$
be the \rep of $\mathbf{G}$ with \hw $\om_1$.

 (1) By Lemma \ref{tt4}, 
$\lam\up$ is an \ei of $s$ on $V$ of \mult equal to that of $\lam.$ As $s$ is  almost cyclic on $V$, we have 
  $\lam\in\{1,-1\}$. The  $1$-eigenspace and the $-1$-eigenspace 
are well known to be non-degenerate and of even dimension.  As $s$ is almost cyclic on $V$, both $1,-1$ cannot be \eis of $s$. \itf $C_{\mathbf{H}}(s)$
is connected, see \cite[Lemma 2.2]{z14}.

 (2)  Let $V_1\subset V$ be  the eigenspace of dimension $m_s(V)$.  We have $V=V_1\oplus V_2$ for some $s$-stable subspace $V_2$ of $V$, and $s$ is cyclic on $V_2$. Then $C_{\mathbf{H}}(h(s))$ contains the subgroup $\mathbf{Y}_1:=\diag(SO(V_1),\Id)$. 

Suppose the contrary, that $s$ is not regular.
Then $C_{G}(s)$ and hence $C_{\mathbf{H}}(h(s))$ contains a unipotent element $u\neq 1$.
If   $\dim V_1=2$ then $u$
is a transvection, however, $\mathbf{H}$ contains no transvection. So $\dim V_1\in\{6,4\}$. 

It is known that $C_{\mathbf {G}}(s)$ is a reductive group  and contains a maximal torus $T$ of $\mathbf {G}$ \cite[Theorem 5.8]{Spr}.

Note that $ \mathbf{G}=Spin_8(F)$ is the universal group of type $D_4$. Let $\si$ be the Steinberg endomorphism of $ \mathbf{G}$ such that $ G={\mathbf G}^{\si}:=\{g\in \mathbf{G}:\si ( g)=g\}$. In particular, $\si (s)=s$ and hence $\si (C_{\mathbf {G}}(s))=C_{\mathbf {G}}(s)$. Then $C_G(s)=
C_{\mathbf {G}}(s)^{\si}$ is a group of maximal rank  in $G$ in terms of \cite{LSS}. Maximal subgroups $M$ of maximal rank in $G$ are determined in \cite{LSS}, and those containing a non-trivial unipotent element are
listed in \cite[Table 5.1]{LSS}. This  tells us that $M$ contains a normal subgroup $M_1$   isomorphic either  to $SL_2(q)\circ SL_2(q^3)$ or to $SL_3(q)$ or to $SU_3(q)$, and  the quotient $M/M_1$ is abelian.  

Let $\mathbf{Y}$ be the semisimple part of $C_{\mathbf{G}}(s).$ Then $h(\mathbf{Y})=\mathbf{Y}_1$.

Suppose first that  $\dim V_1=6$.
Then $\mathbf{Y}\cong D_3\cong A_3\cong SL_4(F)$. 

Obviously,  the restriction of $\si$ to $\mathbf{Y}$ is a Steinberg endomorphism of $\mathbf{Y}$.
 Steinberg endomorphisms of simple algebraic groups  are classified, see \cite[Theorem 22.5 and Table 22.1]{MT}. In particular,  $C_{\mathbf{G}}(s)^{\si}$ contains $\mathbf{Y}^{\si}=SL_4(q)$ or $SU_4(q)$. However, none of these groups is contained in $M$.

Suppose that  $\dim V_1=4$. Set $\mathbf{X}=C_{\mathbf{G} }(\mathbf{Y})^\circ $ 
so $\mathbf{X}V_i=V_i$ for $i=1,2$ and $\mathbf{X}$ acts trivially on $V_1$. Clearly,   
$\si (\mathbf{X})=\mathbf{X}$. Set $\mathbf{K}=\mathbf{X}\mathbf{Y}$; then $\si(\mathbf{K})=\mathbf{K}$ and the rank of $\mathbf{K}$ equals 4. Therefore, $K:=\mathbf{K}^{\si}$ is a subgroup of   maximal rank in $G$ and hence $K\subseteq M$ by the above. 

Observe that $\mathbf{X}\cong\mathbf{Y}\cong Spin_4(F)\cong \mathbf{X}_1\circ \mathbf{X}_2$, where  
$\mathbf{X}_1\cong \mathbf{X}_2\cong SL_2(F)$. 
As $\si (\mathbf{X})=  \mathbf{X}$, we have  either (i) $\si (\mathbf{X}_i)= \mathbf{X}_i$ for $i=1,2$ or (ii) $\si (\mathbf{X}_1)=  \mathbf{X}_2$ and $\si (\mathbf{X}_2)= \mathbf{X}_1$. We show that (ii) does not hold. 

Indeed, by Steinberg's famous theorem \cite[Theorem 21.5]{MT}, a surjective endomorphism of $\mathbf{G} $ is either 
Steinberg or an automorphism. \itf  $\si^i$ is Steinberg for every $i>0$.
In particular, $\si^2$ stabilizes $\mathbf{X}_1$ and $\mathbf{X}_2$. 
Moreover, $\si^3$ acts on maximal tori $T$ of $ \mathbf{H}$ and of $\mathbf{K}$ by sending $t\in T$ to $t^{q^3}$, and hence stabilizes $\mathbf{X}_1$ and $\mathbf{X}_2$ too. For a moment view $\si$ as the abstract group isomorphism $\mathbf{H}\ra \mathbf{H}$. Then $\si=\si^3\si^{-2}$ must stabililze $\mathbf{X}_1$ and $\mathbf{X}_2$. 

So (i) holds.
Similarly,   $\mathbf{Y}\cong Spin_4(F)\cong \mathbf{Y}_1\circ \mathbf{Y}_2$, where 
$\mathbf{Y}_1\cong \mathbf{Y}_2\cong SL_2(F)$ and $\si$
stabililzes both $\mathbf{Y}_1$ and $\mathbf{Y}_2$. Therefore, each group
 $\mathbf{X}^{\si}_1$,  $\mathbf{X}^{\si}_2$,   $\mathbf{Y}^{\si}_1$ and $\mathbf{Y}^{\si}_2$ contains a subgroup isomorphic to  $SL_2(p)$, and hence $K$ contains a subgroup isomorphic to some central product 
of four copies of $SL_2(p)$. However, one easily observes that $M$
contains no subgroup of this kind. This contradiction completes the proof. \enp

The \f result extends Theorem \ref{c99} to finite groups of Lie type. Observe that we use isomorphims $SL_4(q)\cong Spin^+_6(q)$ and 
$SU_4(q)\cong Spin^-_6(q)$ to interpret an \ir  6-dimensional module for
the first group as a twist of the natural module for the second one. Similarly, we use isomorphims $Sp_4(q)\cong SO_5(q)$ to interprete 
the \ir module with \hw $\om_2$ for one of these groups as the natural module for another one.

\begin{theo}\label{gg3} Let G be a universal  quasi-simple group of Lie type and 
$\phi$ a non-trivial \ir F-representation  of G. Let $s\in G$ be a non-regular non-central semisimple element. Suppose that $\phi(g)$ is almost cyclic. Then $G\in\{ SL_{n+1}(q),n>1,SU_{n+1}(q),n>1$, $Spin_{2n+1}(q),n\geq 3$, $q$ odd, $Sp_{2n}(q),n\geq 2$, $Spin^\pm_{2n}(q), n\geq 3\}$ 
and $ \dim\phi$ is 
equal to $n+1, n+1$, $2n+1$, $2n$, $2n$, respectively. 
\end{theo}

\bp It is well known that $\phi$ extends to a \rep of a simply connected simple  algebraic group 
${\mathbf G}$ such that $G={\mathbf G}^{\si}$ for a suitable Steinberg endomorphism of ${\mathbf G}$. We have $s\in {\mathbf G}$.
By  Theorem \ref{c99}, ${\mathbf G}\in\{SL_{n+1}(F),$ $Spin_{2n+1}(F),n\geq 3$, $p>2$, $Sp_{2n}(F),n\geq 2$, $Spin _{2n}(F), n\geq 3\}$ and $\dim \phi=n+1$, $2n+1$, $2n$, $2n$, respectively. 
By the classification theorem of Steinberg endomorphisms of simple algebraic groups, $G={\mathbf G}^{\si}\in\{SL_{n+1}(q),SU_{n+1}(q)$,
$Spin_{2n+1}(q),n> 3$, $q$ odd, $Sp_{2n}(q),n\geq 2$, $Spin^\pm _{2n}(q), n\geq 3, {}^3D_4(q), {}^2C_2(q)\}$. As $s\in G$ is non-regular,
we exclude from this list the groups $SL_{2}(q)\cong SU_{2}(q)$ and  ${}^2C_2(q)$, as they have no non-regular non-central element. The group
${}^3D_4(q)$ must be dropped in view Lemma \ref{d43}. Indeed, 
every \irr of $G={}^3D_4(q)$ of degree 8 is a twist of the natural 
\rep of  $G$  with  
an \au $\tau$, say, of $G$. Obviously, if $\tau(s)=t$ and $\rho(t)$ is almost cyclic then 
so is $\rho^\tau(s)=\rho(\tau(s))$.  So Lemma \ref{d43}(2) extends to every \irr of $G$ of degree 8. 
\enp

\bp[Proof of Theorem {\rm\ref{nr2}}] If $\rho$ is viewed as a \rep into   $GL_n(F)$ then $\rho=\rho_1\oplus\cdots \oplus \rho_k $, where 
$\rho_1\ld \rho_k$ are Galois conjugate \ir \reps of $G$ into $GL_{n/k}(F)$, see \cite[Lemma 2.10.2]{KL}. Each $\rho_i$, $1\leq i\leq k$, extends to a \rep 
$\tau_i:\mathbf{G}\ra GL_{n/k}(F)$, where $\mathbf{G}$ is a simple simply connected \ag such that $G=\mathbf{G}^{\si}$
for a Steinberg endomorphism $\si$ of $\mathbf{G}$ (\cite[\S 2.11]{Hub}). 
Obviously, $\tau_i(g)$   is almost cyclic for $i=1\ld k$. Then 
the result follows from the definition of a nearly natural \rep of $G$ and Theorem \ref{gg3}. \enp

\bigskip

{\it Acknowledgement. }
This paper was initiated while visiting Donna Testerman following on  our collaboration on \cite{TZ2}, \cite{TZ21} and \cite{TZ22}.  The author would like to acknowledge
the hospitality of the \'Ecole Polytechnique Fédérale de Lausanne. At the final stage a part of this work was performed during the author's visit to the Isaacs Newton  Institute for Mathematical Sciences at Cambridge, UK (June-July 2022) in the framework of the research program "Groups, representations and applications".

\bigskip

\newpage
\small
\begin{center}

Table 1: The $p$-restricted  \ir \reps of $G$ whose non-zero weights\\ are of \mult 1
and the zero weight space is not one-dimensional
\bigskip

\begin{tabular}{|l|c|c|c| }

        \hline
         $\,\,\,\,\,\,\,\,\,\,\, G$&highest weight 
&
weight 0 \mult&dimension\\
        \hline
          $A_n,\,\,n>1,\,\,\,\,p\not|\, (n+1)$&$\om_1+\om_n$& $n$&$n^2+2n$
\cr $(n,p)\neq (2,3),\, p\,|\,(n+1)$&$\om_1+\om_n$& $n-1$ &$n^2+2n-1$\cr

 $A_3$,  $p>3$&  $ 2\om_2$& 2  & $20$ \cr \hline

          $B_n, n>2,p\neq 2$& $\om_2$&  $n$& $2n^2+n$\cr
\,\,\,\,\, $p\,\,|\,(2n+1)$&$2\om_1$&$n$ &$2n^2+3n-1 $ \cr
\,\,\,\,\, $p\not|\,(2n+1)$&$2\om_1$& $n+1$&$2n^2+3n$\cr
 $B_2$, $ p\neq 2$& $ 2\om_2 $&  $2$&10 \cr
 \,\,\,\,\, $ p\neq 2,5$& $ 2\om_1 $&  $2$&14 \cr
\hline



$C_n$, $n>2,$&$2\om_1$ &$n$&$2n^2+n$\cr

\,\,\,\,\, $(n,p)\neq(3,3)$&$\om_2$& $n-1$ if $p\not|n$&$2n^2-n-1$\cr

&& $n-2$ if $p|n$,\,\,&$2n^2-n-2$\cr

$C_2$, $p\neq 2$&$2\om_1$&$2$&10\cr

\,\,\,\,\, $p\neq 2,5$&$2\om_2$&$2$&14 \cr

$C_4$, $p\neq 2,3$&$\om_4$&$2$&36\cr

\hline

 $D_n$, $n>3,p\neq 2$&$2\om_1$& $ n-2$ if $p|n$, & $2n^2+n-2$ if $p|n$,\\
&&$n-1$ if
$p\not|n$&  $2n^2+n-1$ if $p\not|n$\cr
 & $\om_2$& $n$&$2n^2-n-1$\cr
\,\,\,\,\, \,\,\,\,\, $p=2$& $\om_2$&
$n-(2,n)$ &$2n^2-n-(2,n)$\cr

\hline

          $E_6,$ \,\,\,\,$p=3$&$\om_2$&5  & 77  \cr
\,\,\,\,\, \,\,\,\,\,\,\,\,$p\neq 3$&& 6 &78\cr
\hline $E_7$, \,\,\,\,$p\neq 2$&$\om_1$ &7&133 \cr
\,\,\,\,\, \,\,\,\,\,\,\,$p= 2$&&6& 132 \\
 \hline
 $E_8$&$\om_8$ & 8&248\cr
\hline
          $F_4,$\,\,\,\,\,\,$p=2$&$\om_1$ &2 &26    \cr
\,\,\,\,\, \,\,\,\,\,\,\,$p\neq 2$&&4&52\cr
\,\,\,\,\, \,\,\,\,\,  $p\neq 3$& $\om_4$ &2 &26\cr
\hline
          $ G_2$, $\,\,\,p\neq 3$&$\om_2$&2&14
             \\
\hline
    \end{tabular}

\end{center}

\vspace{10pt}
\begin{center}


\newpage

Table 2: Non-trivial $p$-restricted $G$-modules 
 whose weights are one-dimensional
\vskip1cm

\small{

\bigskip
\begin{tabular}{|l|c|}

        \hline
         $\,\,\,\,\,\,\,\,\,\,\,  G$& highest weight  
\\

  \hline
$A_1$&$a\om_1$, $\,\,\,\, 1\leq a<p$\cr
        \hline
          $A_n, n>1$&$a\om_1$, $\,\,b\om_n$, $\,\, 0\leq a,b<p$, $\,\,\om_i$, $1<i< n$
\cr &$c\om_i+(p-1-c)\om_{i+1}$,  $\,\,i=1\ld n-1,$ $\,\, 0\leq c<p$ \cr
\hline
          $B_n$, $n\geq 2$&$\om_1,\om_n$  \cr
\hline
$B_2$, $p\neq 2,3$&$\om_1,\om_2, \frac{p-3}{2}\om_1+\om_2,\frac{p-1}{2}\om_1$
\cr\hline
          $C_n$, $n>1$, $p=2$&$\om_1,\om_n$\cr  \hline
$C_n$, $n\geq 3$, $p\neq   2$&$\om_1,
\om_{n-1}+\frac{p-3}{2}\om_n,\frac{p-1}{2}\om_n$\cr
\hline $C_2$, $p\neq 2,3$&$\om_1,\om_2,
\om_1+\frac{p-3}{2}\om_2,\frac{p-1}{2}\om_2$\cr \hline
$C_3$, $p\neq   2$&$\om_3$\cr
\hline
$D_n$, $n>3$&$\om_1,\om_{n-1}, \om_n$\cr \hline

          $E_6$&$\om_1$, $\om_6$\cr
\hline
          $E_7$&$\om_7$  \cr
\hline
          $F_4 $, $p=3$&$\om_4$ \cr
\hline
          $ G_2$, $p\neq 3$&$\om_1$
             \\
\hline $ G_2$, $p= 3$&$\om_1,\om_2$
             \\
\hline   \end{tabular}
}
\end{center}

\bigskip
  Department of Mathematics, University of Brasilia, Brasilia-DF,
70910-900 Brazil 

{\small
e-mail: alexandre.zalesski$@$gmail.com}

\end{document}